\documentclass[oneside,reqno,english]{amsart}
\usepackage[T1]{fontenc}
\usepackage[latin9]{inputenc}
\usepackage{babel}
\usepackage{url}
\usepackage{amstext}
\usepackage{amsthm}
\usepackage{amssymb}
\usepackage[unicode=true,pdfusetitle,
 bookmarks=true,bookmarksnumbered=false,bookmarksopen=false,
 breaklinks=false,pdfborder={0 0 1},backref=false,colorlinks=false]
 {hyperref}

\makeatletter
\numberwithin{equation}{section}
\theoremstyle{remark}
\newtheorem{rem}{\protect\remarkname}
\theoremstyle{plain}
\newtheorem{lem}{\protect\lemmaname}
\theoremstyle{definition}
\newtheorem{defn}{\protect\definitionname}
\theoremstyle{plain}
\newtheorem{thm}{\protect\theoremname}
\theoremstyle{definition}
 \newtheorem{example}{\protect\examplename}
\theoremstyle{plain}
\newtheorem{cor}{\protect\corollaryname}
\theoremstyle{plain}
\newtheorem{prop}{\protect\propositionname}

\makeatother

\providecommand{\corollaryname}{Corollary}
\providecommand{\definitionname}{Definition}
\providecommand{\examplename}{Example}
\providecommand{\lemmaname}{Lemma}
\providecommand{\propositionname}{Proposition}
\providecommand{\remarkname}{Remark}
\providecommand{\theoremname}{Theorem}

\begin{document}
\title[$L_{p}-$ solvability of integro-Differential equations]{\noindent on $L_{p}-$ theory for integro-differential operators
with spatially dependent coefficients }
\author{\noindent sutawas janreung, Tatpon siripraparat, and chukiat saksurakan{*}}
\begin{abstract}
The parabolic integro-differential Cauchy problem with spatially dependent
coefficients is considered in generalized Bessel potential spaces
where smoothness is defined by Lévy measures with O-regularly varying
profile. The coefficients are assumed to be bounded and Hölder continuous
in the spatial variable. Our results can cover interesting classes
of Lévy measures that go beyond those comparable to $dy/\left|y\right|^{d+\alpha}.$ 
\end{abstract}

\address{\noindent Department of Social and Applied Science, College of Industrial
Technology, King Mongkut\textquoteright s University of Technology
North Bangkok, 1518 Pracharat 1 Road, Wongsawang, Bangsue, Bangkok
10800 Thailand.}
\email{sutawas.j@cit.kmutnb.ac.th}
\address{\noindent Department of Social and Applied Science, College of Industrial
Technology, King Mongkut\textquoteright s University of Technology
North Bangkok, 1518 Pracharat 1 Road, Wongsawang, Bangsue, Bangkok
10800 Thailand.}
\email{tatpon.s@cit.kmutnb.ac.th }
\address{Center of Sustainable Energy and Engineering Materials (SEEM), College
of Industrial Technology, King Mongkut\textquoteright s University
of Technology North Bangkok, 1518 Pracharat 1 Road, Wongsawang, Bangsue,
Bangkok 10800 Thailand.}
\email{chukiat.s@cit.kmutnb.ac.th}
\thanks{\noindent This research was funded by College of Industrial Technology,
King Mongkut's University of Technology North Bangkok (Grant No. Res-CIT320/2023)}
\thanks{{*}Corresponding author, email: chukiat.s@cit.kmutnb.ac.th}
\maketitle

\section{Introduction}

Let $\sigma\in\left(0,2\right)$ and $\mathfrak{A}^{\sigma}$ be the
class of all non-negative measures $\nu$ on $\mathbf{R}_{0}^{d}=\mathbf{R}^{d}\backslash\left\{ 0\right\} $
such that $\int\left(\left\vert y\right\vert ^{2}\wedge1\right)\nu\left(dy\right)<\infty$
and 
\[
\sigma=\inf\left\{ \alpha<2:\int_{\left\vert y\right\vert \leq1}\left\vert y\right\vert ^{\alpha}\nu\left(dy\right)<\infty\right\} .
\]

Let $\nu\in\mathfrak{A}^{\sigma}$, we consider the non-local parabolic
Cauchy problem with spatially dependent coefficients, 
\begin{eqnarray}
du\left(t,x\right) & = & \left[L^{m,\nu}u\left(t,x\right)-\lambda u\left(t,x\right)+f\left(t,x\right)\right]dt\label{eq:mainEq}\\
u\left(0,x\right) & = & 0,\left(t,x\right)\in E=\left[0,T\right]\times\mathbf{R}^{d}\nonumber 
\end{eqnarray}
with $\lambda\geq0$ and integro-differential operator $L^{m,\nu}\varphi\left(x\right)=L_{t,x}^{m,\nu}\varphi\left(x\right)=L_{t,z}^{m,\nu}|_{z=x}\varphi\left(x\right)$
where $L_{t,z}^{m,\nu}$ is given by
\[
L=L_{t,z}^{m,\nu}=\int\left[\varphi\left(x+y\right)-\varphi\left(x\right)-\chi_{\sigma}\left(y\right)y\cdot\nabla\varphi\left(x\right)\right]m\left(t,z,y\right)\nu\left(dy\right),
\]
for $\varphi\in C_{0}^{\infty}\left(\mathbf{R}^{d}\right)$ and $\chi_{\sigma}\left(y\right)=0$
if $\sigma\in\left[0,1\right),\chi_{\sigma}\left(y\right)=1_{\left\{ \left\vert y\right\vert \leq1\right\} }\left(y\right)$
if $\sigma=1$ and $\chi_{\sigma}\left(y\right)=1$ if $\sigma\in\left(1,2\right).$ 

\medskip

For $\nu\in\mathfrak{A}^{\sigma}$, we define its radial distribution
function 
\[
\delta\left(r\right)=\delta_{\nu}\left(r\right)=\nu\left(x\in\mathbf{R}^{d}:\left\vert x\right\vert >r\right),r>0,
\]
and

\[
w\left(r\right)=w_{\nu}\left(r\right)=\delta_{\nu}\left(r\right)^{-1},r>0.
\]

Let $w$ be continuous and positive. Our main assumption is that $w_{\nu}\left(r\right)$
is an O-RV function at both infinity and at zero. That is 
\[
r_{1}\left(x\right)=\limsup_{\epsilon\rightarrow0}\frac{w_{\nu}\left(\epsilon x\right)}{w_{\nu}\left(\epsilon\right)}<\infty,\hspace{1em}r_{2}\left(x\right)=\limsup_{\epsilon\rightarrow\infty}\frac{w_{\nu}\left(\epsilon x\right)}{w_{\nu}\left(\epsilon\right)}<\infty,x>0.
\]

\noindent The regular variation functions were introduced in \cite{Kar}
and used in tauberian theorems which were also extended to O-RV functions
(e.g. \cite{AA,BGT} and references therein.) This assumption yields
several nice properties regarding asymptotic integrals and scaling
of $w_{\nu}$ (see Appendix) which are convenient for verifying Hörmander
conditions in Calderón--Zygmund theory of singular integrals. It
is worth noting that this type of scaling properties are widely used
in estimates of heat kernel associated to Markov processes (e.g. \cite{KR,CKK})
and also in relevant literature studying (\ref{eq:mainEq}) which
we now elaborate. 

The $L_{p}-$regularity for (\ref{eq:mainEq}) when $\nu\left(dy\right)$
is comparable to $dy/\left|y\right|^{d+\alpha}$ and $m$ is not a
constant was first studied by \cite{MPr92} in which $m$ is Hölder
continuous in $x$, homogeneous of order $0$ and smooth in $y$.
In \cite{DK12}, the authors studied the elliptic problem when $m\left(y\right)$
is independent of the spatial variable. They showed the boundedness
of $L^{\nu}$ by using Hölder norm estimates and only required $m\left(y\right)$
to be bounded and measurable. Later in \cite{MPr14}, the authors
of \cite{MPr92} weaken their assumptions and requires regularity
for $y$ to hold for a lower-bound coefficient $m_{0}\left(t,y\right)\leq m\left(t,x,y\right)$
instead of $m$. Their results cover $p>d/\beta$ where $\beta\in\left(0,1\right)$
is the order of the modulus of continuity (see the definition in Assumption
\textbf{C} below.) A related martingale problem was also studied.
The restriction to large $p$ stems directly from Sobolev embedding
theorem. Later, results in \cite{MPr14} were improved by \cite{DK}
to cover all $p>1$ and weighted $L_{p}$ spaces based on the Rubio
de Francia extrapolation theorem and weighted continuity of operator
$L^{\nu}$. We also mention that a relevant result for all $p>1$
was derived in \cite{Zha12} where less regularity namely $\beta=0$
is required in $x$, however, the modulus of continuity is also required
in $y$. 

From a broader perspective, there has been considerable research on
(\ref{eq:mainEq}) when $m$ is a constant with more a generalized
measure $\nu\left(dy\right)$ other than $dy/\left|y\right|^{d+\alpha}$.
As an example, a measure $\nu\in$$\mathfrak{A}$$^{\sigma}$of the
following form defined in radial and angular coordinates $r=\left\vert y\right\vert ,w=y/r,$
as 
\begin{equation}
\nu\left(\Gamma\right)=\int_{0}^{\infty}\int_{\left\vert z\right\vert =1}\chi_{\Gamma}\left(rz\right)a\left(r,z\right)j\left(r\right)r^{d-1}S\left(dz\right)dr,\Gamma\in\mathcal{B}\left(\mathbf{R}_{0}^{d}\right),\label{ex0}
\end{equation}
where $S\left(dz\right)$ is a finite measure on the unit sphere on
$\mathbf{R}^{d}$ has been investigated. Specifically, the case of
$a=1,j\left(r\right)=r^{-d-\sigma}$ with relatively general non-degenerate
$S$ has been studied in \cite{Zha13} while the case of $S\left(dz\right)=dz$
being a Lebesgue measure on the unit sphere in $\mathbf{R}^{d}$ with
some technical assumptions on more general $j\left(r\right)$ has
been studied in \cite{KK16}. More recently, an anisotropic measure
with potentially different orders of the following form with some
technical assumptions on $j_{i}\left(y_{i}\right)$ was studied in
\cite{CKP},
\begin{equation}
\nu\left(dy\right)=\sum_{i=1}^{d}j_{i}\left(y_{i}\right)dy_{i}\epsilon_{0}^{i}\left(dy_{1},...,dy_{i-1},dy_{i+1},...,dy_{d}\right)\label{ex1}
\end{equation}

\noindent where $\epsilon_{0}^{i}$ is the Dirac measure in $\mathbf{R}^{d-1}$
concentrated at the center. 

The authors of \cite{MPh3} has proposed a set of assumptions that
can cover (\ref{ex0}) when both $j\left(r\right)$ can be more general
than $r^{-d-\sigma}$ and $S\left(dz\right)$ can be more general
than a Lebesgue measure. To the best of our knowledge this set of
assumptions remain the only one that satisfy such requirement. These
assumptions can also cover a partial case of (\ref{ex1}) when $j_{i}$
are the same for all $i$ such as when $j_{i}\left(y_{i}\right)=\left|y_{i}\right|^{-1-\alpha}$(see
Example \ref{example:aniso} below.)

There has been limited research on (\ref{eq:mainEq}) with generalized
measure $\nu\left(dy\right)$ with spatially dependent coefficient
$m\left(t,x,y\right)$. As far as we know, even in the case that $m$
only depends on $t$ and $y$ only \cite{KKK} deal with non $\alpha-$stable
measures. In fact, they allow $\nu_{t}\left(dy\right)$ to depend
on time directly. Motivated by this gap in research we investigate
(\ref{eq:mainEq}) with the set assumptions from \cite{MPh3} due
to its relative generality (see Examples.) We adapt assumptions and
some methods from \cite{MPr92,DK} using freezing coefficient and
continuation by parameter arguments. However, since $\nu\left(dy\right)$
is more general than in our case we need a more generalized analysis
on our transitional probabilities. Our result is new even when $m$
only depends on $t$ and $y$. 

We would like to mention another direction of research which studies
(\ref{eq:mainEq}) with the Caputo derivative $\partial_{t}^{\alpha}$
for some $\alpha\in\left(0,1\right)$ instead of the standard derivative
for the time variable. The reader may consult \cite{KP,DL23} regarding
this topic. 

\section{Notations, Assumptions, Main Results, Examples}

\subsection{Notations}

\noindent The following notation will be used in the paper.
\begin{itemize}
\item We write $E=\left[0,T\right]\times\mathbf{R}^{d}.$ 
\item We write $f\asymp g$ if there exists $c>0$ so that $c^{-1}g\left(x\right)\leq f\left(x\right)\leq cg\left(x\right)$
for all $x$ in the domain of $f$ and $g$. We will also use $f\preceq g$
(resp. $f\succeq g$) when $f\leq cg$ (resp. $f\geq cg$) with some
$c>0$ for all $x$ in the domain of $f,g.$ 
\item For any real-valued function $f$, we denote $f^{+}=f\vee0,f^{-}=f\wedge0.$
\item Let $\mathbf{N}=\{1,2,\ldots\},\mathbf{N}_{0}=\left\{ 0,1,\ldots\right\} ,\mathbf{R}_{0}^{d}=\mathbf{R}^{d}\backslash\{0\}.$
If $x,y\in\mathbf{R}^{d},$\ we write 
\[
x\cdot y=\sum_{i=1}^{d}x_{i}y_{i},\,|x|=\sqrt{x\cdot x}.
\]
\item We denote by $C_{0}^{\infty}\left(\mathbf{R}^{d}\right)$ the set
of all infinitely differentiable functions on $\mathbf{R}^{d}$ with
compact support. As usual, we denote $\mathcal{S}\left(\mathbf{R}^{d}\right)$
the Schwartz function space on $\mathbf{R}^{d}$. 
\item We denote the partial derivatives in $x$ of a function $u\left(t,x\right)$
on $\mathbf{R}^{d+1}$ by $\partial_{i}u=\partial u/\partial x_{i}$,
$\partial_{ij}^{2}u=\partial^{2}u/\partial x_{i}\partial x_{j}$,
etc.; $Du=\nabla u=\left(\partial_{1}u,\ldots,\partial_{d}u\right)$
denotes the gradient of $u$ with respect to $x$; for a multiindex
$\gamma\in\mathbf{N}_{0}^{d}$ we denote 
\[
D_{x}^{\gamma}u\left(t,x\right)=\frac{\partial^{|\gamma|}u\left(t,x\right)}{\partial x_{1}^{\gamma_{1}}\ldots\partial x_{d}^{\gamma_{d}}}.
\]
If $\left|\gamma\right|=1$, we use $D_{x}u\left(t,x\right)$ to denote
any first order derivative.
\item For a function one variable we denote by $f^{\left(j\right)}$ the
$j$-th order derivative of $f.$
\item For $\alpha\in\left(0,2\right]$ and a function $u\left(t,x\right)$
on $\mathbf{R}^{d+1}$, we write 
\[
\partial^{\alpha}u\left(t,x\right)=-\mathcal{F}^{-1}\left[|\xi|^{\alpha}\mathcal{F}u\left(t,\xi\right)\right]\left(x\right),
\]
where the Fourier transform and the inverse Fourier transform are
given by
\begin{align*}
\mathcal{F}h\left(t,\xi\right) & =\hat{h}\left(\xi\right)=\int_{\mathbf{R}^{d}}\mathrm{e}^{-i2\pi\xi\cdot x}h\left(t,x\right)dx,\\
\mathcal{F}^{-1}h\left(t,\xi\right) & =\int_{\mathbf{R}^{d}}\mathrm{e}^{i2\pi\xi\cdot x}h\left(t,\xi\right)d\xi,
\end{align*}
respectively.
\item For a radial function $f:\mathbf{R}^{d}\rightarrow\mathbf{R}$, we
write with a slight abuse of notation $f\left(y\right)=f\left(\left|y\right|\right)=f\left(r\right)$
when $\left|y\right|=r.$
\item We shall write throughout this paper that $L^{\nu}=L^{1,\nu}$ for
the case of $m=1.$
\item For $\nu\in$$\mathfrak{A}$$^{\sigma}$, we denote $Z_{t}^{\nu},t\geq0,$
the Lévy process associated to $L^{\nu}$, i.e., $Z_{t}^{\nu}$ is
cadlag with independent increments and its characteristic function
\[
\mathbf{E}e^{i2\pi\xi\cdot Z_{t}^{\nu}}=\exp\left\{ \psi^{\nu}\left(\xi\right)t\right\} ,\xi\in\mathbf{R}^{d},t\geq0.
\]
\item The letters $C=C\left(\cdot,\ldots,\cdot\right),N=N\left(\cdot,...,\cdot\right)$
and $c=c\left(\cdot,\ldots,\cdot\right)$ with and without subscriptions
denote constants depending only on quantities appearing in parentheses.
In a given context the same letter will generally be used to denote
different constants depending on the same set of arguments.
\item We will also use these transformations of $\nu$ and $m$.

\noindent (i) $\nu^{\ast}\left(dy\right)=\nu\left(-dy\right)$ and
$m^{\ast}\left(t,y\right)=m\left(t,-y\right),t\geq0,y\in\mathbf{R}^{d}$

\noindent (ii) $\nu_{sym}\left(dy\right)=\frac{\nu\left(dy\right)+\nu\left(-dy\right)}{2},$
a symmetrization of $\nu$

\noindent (iii) $\tilde{\nu}_{R}\left(dy\right)=w_{\nu}\left(R\right)\nu\left(Rdy\right),R>0$

\noindent (iv) $m_{R}\left(t,y\right)=m\left(t,Ry\right),t\geq0,R>0$
\end{itemize}

\subsection{Assumptions }

\noindent We will study (\ref{eq:mainEq}) under the same set of assumptions
from \cite{MPh3} due to its relative generality which we justify
in the next section (see Examples.) 

We set for $\nu\in\mathfrak{A}^{\sigma}$ 
\[
\delta\left(r\right)=\delta_{\nu}\left(r\right)=\nu\left(x\in\mathbf{R}^{d}:\left|x\right|>r\right),r>0
\]

\noindent 
\[
w\left(r\right)=w_{\nu}\left(r\right)=\delta_{\nu}\left(r\right)^{-1},r>0.
\]

Our main assumption is that $w=w_{\nu}\left(r\right)$ is an O-RV
function at both infinity and at zero. That is 
\begin{equation}
r_{1}\left(x\right)=\limsup_{\delta\rightarrow0}\frac{w\left(\delta x\right)}{w\left(\delta\right)}<\infty,\hspace{1em}r_{2}\left(x\right)=\limsup_{\delta\rightarrow\infty}\frac{w\left(\delta x\right)}{w\left(\delta\right)}<\infty,x>0.\label{eq:def_r}
\end{equation}

Due to Theorem 2 in \cite{AA}, the following limit exist: 
\begin{equation}
p_{1}=p_{1}^{w}=\lim_{\delta\rightarrow0}\frac{\log r_{1}\left(\delta\right)}{\log\delta}\leq q_{1}=q_{1}^{w}=\lim_{\delta\rightarrow\infty}\frac{\log r_{1}\left(\delta\right)}{\log\left(\delta\right)}\label{1}
\end{equation}
and 
\begin{equation}
p_{2}=p_{2}^{w}=\lim_{\delta\rightarrow0}\frac{\log r_{2}\left(\delta\right)}{\log\delta}\leq q_{2}=q_{2}^{w}=\lim_{\delta\rightarrow\infty}\frac{\log r_{2}\left(\delta\right)}{\log\left(\delta\right)}.\label{2}
\end{equation}
Note that $p_{1}\leq\sigma\leq q_{1}$ (see \cite{MF}.) We assume
throughout this paper that $p_{1},p_{2},q_{1},q_{2}>0$. The numbers
$p_{1},p_{2}$ are called lower indices and $q_{1},q_{2}$ are called
upper indices of O-RV function. When there is no chance of confusion,
for a function $f$ which is both O-RV at zero and infinity, we always
denote its lower and upper indices at zero by $p_{1},q_{1}$ respectively,
and its lower and upper indices at infinity by $p_{2},q_{2}$ respectively.
We also write $r_{1},r_{2}$ as given in (\ref{eq:def_r}). We denote
by $p_{1}^{f},p_{2}^{f},q_{1}^{f},q_{2}^{f},r_{1}^{f},r_{2}^{f}$
when an explicit reference to $f$ is needed. 

We assume throughout this paper without further mentioning that for
$\nu\in\mathfrak{A}^{1}$ ($\sigma=1$) 
\begin{eqnarray*}
\int_{r<\left\vert y\right\vert \leq R}y\nu\left(dy\right) & = & 0\text{ for all }0<r<R<\infty,\text{ }
\end{eqnarray*}

\noindent and the coefficient $m$ is always measurable. 

\noindent Our main assumptions are

\medskip

\noindent \textbf{Assumption A.}

\noindent For $i=1,2$ 
\begin{eqnarray*}
0 & < & p_{i}\leq q_{i}<1\text{ if }\sigma\in\left(0,1\right),0<p_{i}\leq1\leq q_{i}<2\text{ if }\sigma=1,\\
1 & < & p_{i}\leq q_{i}<2\text{ if \ensuremath{\sigma}}\in\left(1,2\right).
\end{eqnarray*}

\medskip

\noindent \textbf{Assumption B.}
\[
\inf_{R\in\left(0,\infty\right),\left\vert \hat{\xi}\right\vert =1}\int_{\left\vert y\right\vert \leq1}\left\vert \hat{\xi}\cdot y\right\vert ^{2}\tilde{\nu}_{R}\left(dy\right)>0.
\]

\medskip

\noindent The next assumption will be used for the coefficient $m=m\left(t,x,y\right),\left(t,x\right)\in E,y\in\mathbf{R}^{d}$.
We shall assume that $m$ is bounded and Hölder continuous with the
modulus $\kappa$ in the following sense: 

\medskip

\noindent \textbf{Assumption C. }

\noindent (i) For all $x,y\in\mathbf{R}^{d}$ and $t\in\left[0,T\right]$,
\[
0<k\leq m\left(t,x,y\right)\leq K.
\]

\noindent (ii) There exists $\beta\in\left(0,1\right)$ and a continuous
increasing function $\kappa\left(\tau\right),\tau>0$ such that 
\[
\left|m\left(t,x_{1},y\right)-m\left(t,x_{2},y\right)\right|\leq\kappa\left(\left|x_{1}-x_{2}\right|\right)
\]

\noindent for $t\in\left[0,T\right],x_{1},x_{2},y\in\mathbf{R}^{d}$
and 

\noindent 
\[
\int_{\left|y\right|\leq1}\kappa\left(\left|y\right|\right)\frac{dy}{\left|y\right|^{d+\beta}}<\infty.
\]

\noindent In addition if $\sigma=1$ then for any $\left(t,x\right)\in E$
and $0<r<R<\infty,$

\noindent 
\[
\int_{r\leq\left|y\right|\leq R}ym\left(t,x,y\right)\nu\left(dy\right)=0.
\]

\medskip
\begin{rem}
Assumption \textbf{C }implies $\lim_{\delta\rightarrow0}\kappa\left(\delta\right)\delta^{-\beta}=0$
(see \cite[Remark 2.2]{DK}.)
\end{rem}
\begin{rem}
\noindent We will pay attention to a special case when $\nu$ is an
isotropic unimodal measure namely $\nu\left(dy\right)=j_{d}\left(\left|y\right|\right)dy$
where $j_{d}\left(r\right)\asymp r^{-d}\gamma\left(r\right)^{-1},r>0$
and $\gamma$ is an O-RV function. We note that in this case Assumption
$\textbf{B}$ can be easily verified. To exploit gradient estimates
of transitional probability provided in \cite[Theorem 1.5]{KR}, we
will impose the following assumption on $j_{d}$. 

\medskip
\end{rem}
\noindent \textbf{Assumption D.} $j_{d}:\left(0,\infty\right)\rightarrow\left(0,\infty\right)$
is differentiable and 

\noindent (i) $j_{d}\left(r\right)$ and $-\frac{1}{r}j_{d}^{\left(1\right)}\left(r\right)$
are non-increasing.

\noindent (ii) There exists a non-decreasing continuous O-RV function
$\gamma:\left(0,\infty\right)\rightarrow\left(0,\infty\right)$ satisfying
Assumption \textbf{A} and constants $c_{1},c_{2}>0$ such that for
$n=0,1$
\[
c_{1}r^{-d-n}\gamma\left(r\right)^{-1}\leq\left(-1\right)^{j}j_{d}^{\left(n\right)}\left(r\right)\leq c_{2}r^{-d-n}\gamma\left(r\right)^{-1},r>0.
\]

We now mention a useful Lemma regarding moment estimates of $\nu$
with scaling.

\medskip
\begin{lem}
\label{lem:alpha12_integral}(\cite[Lemma 4]{MPh3}) Let $\nu\in\mathfrak{A}^{\sigma}$,$w=w_{\nu}$
be an O-RV at zero and infinity with indices $p_{1},q_{1},p_{2},q_{2}$
defined in (\ref{1}), (\ref{2}). Then for any $\alpha_{1}>q_{1}\vee q_{2}$
and $0<\alpha_{2}<p_{1}\wedge p_{2}$, there is $C=C\left(w,\alpha_{1},\alpha_{2}\right)>0$
such that 
\[
\int_{\left\vert y\right\vert \leq1}\left\vert y\right\vert ^{\alpha_{1}}\tilde{\nu}_{R}\left(dy\right)+\int_{\left\vert y\right\vert >1}\left\vert y\right\vert ^{\alpha_{2}}\tilde{\nu}_{R}\left(dy\right)\leq C,R>0,
\]

\noindent where $\tilde{\nu}_{R}\left(dy\right)=w_{\nu}\left(R\right)\nu\left(Rdy\right),R>0.$
\end{lem}
With Lemma \ref{lem:alpha12_integral} in mind, we also introduce
a notation for an auxiliary measure $\pi\in\mathfrak{A}^{\sigma}$
whose scaling is controlled by $\nu\in\mathfrak{A}^{\sigma}$ (with
the same $\sigma$.)

Let

\noindent 
\[
\alpha_{1}>q_{1}^{w_{\nu}}\vee q_{2}^{w_{\nu}},0<\alpha_{2}<p_{1}^{w_{\nu}}\wedge p_{2}^{w_{\nu}},
\]
and 
\begin{eqnarray*}
\alpha_{2} & > & 1\text{ if }\sigma\in\left(1,2\right),\\
\alpha_{1} & \leq & 1\text{ if }\sigma\in\left(0,1\right),\alpha_{1}\leq2\text{ if }\sigma\in\left[1,2\right).
\end{eqnarray*}

\noindent if for some $M>0$,
\[
\int_{\left\vert y\right\vert \leq1}\left\vert y\right\vert ^{\alpha_{1}}\tilde{\pi}_{R}\text{\ensuremath{\left(dy\right)}}+\int_{\left\vert y\right\vert >1}\left\vert y\right\vert ^{\alpha_{2}}\tilde{\pi}_{R}\left(dy\right)\leq M,R>0.
\]

\noindent Then we write 
\[
\pi\prec\nu\left(M,\alpha_{1},\alpha_{2}\right)
\]
 and simply $\pi\prec\nu$ if $\pi\prec\nu\left(M,\alpha_{1},\alpha_{2}\right)$
for some $M,\alpha_{1},\alpha_{2}$. 

\noindent 

\subsection{Function Spaces and Main Results}

It is well known at this point that in order to capture optimal regularity
consistent with the classical theory of non-degenerate parabolic equations,
one need to study (\ref{eq:mainEq}) in a generalized Bessel potential
space defined with respect to the symbol of $L^{\nu}$ itself. As
far as we know, these spaces were first used for a non-local parabolic
equation in \cite{MPh3} where some basic properties and characterizations
on the scale were also rigorously investigated. Similar generalized
spaces were also later used in \cite{KP,CKP,KK23} under varying assumptions
about $\nu$. 

We use standard notation $\mathcal{S}\left(\mathbf{R}^{d}\right)$
for the Schwartz space of real valued rapidly decreasing functions,
and $L_{p}=L_{p}\left(\mathbf{R}^{d}\right),p\in\left[1,\infty\right)$
for the space of locally integrable functions with the finite norm
\[
\left|f\right|_{L_{p}}=\left(\int\left|f\left(x\right)\right|^{p}dx\right)^{1/p}.
\]

On $E=\left[0,T\right]\times\mathbf{R}^{d}$, we write $L_{p}\left(E\right)$
the space of locally integrable functions with the finite norm 
\[
\left(\int_{0}^{T}\left|f\right|_{L_{p}}^{p}dt\right)^{1/p}.
\]

Let 
\[
Jf=J_{\nu}f=\left(I-L^{\nu_{sym}}\right)f=f-L^{\nu_{sym}}f,f\in\mathcal{S}\left(\mathbf{R}^{d}\right).
\]
For $s\in\mathbf{R}$ set 
\[
J^{s}f=\left(I-L^{\nu_{sym}}\right)^{s}f=\mathcal{F}^{-1}\left[\left(1-\psi^{\nu_{sym}}\right){}^{s}\hat{f}\right],f\in\mathcal{S}\left(\mathbf{R}^{d}\right).
\]

\noindent 
\[
L^{\nu;s}f=\mathcal{F}^{-1}\left(-\left(-\psi^{\nu_{sym}}\right)^{s}\hat{f}\right),f\in\mathcal{S}\left(\mathbf{R}^{d}\right).
\]

\noindent Note that $L^{\nu;1}f=L^{\nu_{sym}}f,f\in\mathcal{S}\left(\mathbf{R}^{d}\right).$

For $p\in\left[1,\infty\right),s\in\mathbf{R,}$ we define, following
\cite{FJS}, the Bessel potential space $H_{p}^{\nu;s}=H_{p}^{\nu;s}\left(\mathbf{R}^{d}\right)$
as the closure of $\mathcal{S}\left(\mathbf{R}^{d}\right)$ in the
norm 
\begin{eqnarray*}
\left\vert f\right\vert _{H_{p}^{\nu;s}} & = & \left\vert J^{s}f\right\vert _{L_{p}}=\left\vert \mathcal{F}^{-1}\left[\left(1-\psi^{\nu_{sym}}\right){}^{s}\hat{f}\right]\right\vert _{L_{p}}\\
 & = & \left\vert \left(I-L^{\nu_{sym}}\right)^{s}f\right\vert _{L_{p}},f\in\mathcal{S}\left(\mathbf{R}^{d}\right).
\end{eqnarray*}

Clearly, $H_{p}^{\nu;0}\left(\mathbf{R}^{d}\right)=L_{p}\left(\mathbf{R}^{d}\right)$
and $H_{p}^{\nu;s}$ is a Banach space.
\begin{rem}
\noindent \label{rem:gen_bessel} We comment on some basic facts regarding
$H_{p}^{\nu;s}.$

\noindent (i) By \cite[Theorem 2.3.1]{FJS}, if $p\in\left(1,\infty\right),s<t$
then $H_{p}^{\nu;t}\left(\mathbf{R}^{d}\right)\subseteq H_{p}^{\nu;s}\left(\mathbf{R}^{d}\right)$
is continuously embedded. 

\noindent (ii) By \cite[Theorem 2.2.7]{FJS}, for $s\geq0,p\in\left[1,\infty\right),$
the norm $\left\vert f\right\vert _{H_{p}^{\nu;s}}$ is equivalent
to 
\begin{align*}
\left\vert f\right\vert _{H_{p}^{\nu;s}} & =\left\vert f\right\vert _{L_{p}}+\left\vert \mathcal{F}^{-1}\left[\left(-\psi^{\nu_{sym}}\right){}^{s}\mathcal{F}f\right]\right\vert _{L_{p}}\\
 & =\left\vert f\right\vert _{L_{p}}+\left\vert L^{\nu;s}f\right\vert _{L_{p}}.
\end{align*}
\end{rem}
We define $H_{p}^{\nu;s}\left(E\right)$ the space of locally integrable
functions with the finite norm, 
\[
\left\vert f\right\vert _{H_{p}^{\nu;s}\left(E\right)}=\left(\int_{0}^{T}\left|f\right|_{H_{p}^{\nu;s}}^{p}\right)^{1/p}.
\]

For the special case where $s=1$, we denote for simplicity of notation
$H_{p}^{\nu}=H_{p}^{\nu;1}$ and respectively $H_{p}^{\nu}\left(E\right)=H_{p}^{\nu;1}\left(E\right)$.
In this case, $\left\vert f\right\vert _{H_{p}^{\nu}}=\left\vert f\right\vert _{L_{p}}+\left\vert L^{\nu_{sym};1}f\right\vert _{L_{p}}.$
As we will show later in Lemma \ref{lem:(Continuity-of-Operators)}
(see also Remark \ref{rem:norm_equi}) that under Assumptions \textbf{A}
and \textbf{B}, $\left\vert L^{\nu_{sym};1}f\right\vert _{L_{p}}\asymp\left\vert L^{\nu}f\right\vert _{L_{p}}$
which implies that the norm $\left\vert f\right\vert _{H_{p}^{\nu}}$
is in fact equivalent to 
\[
\left\vert f\right\vert _{H_{p}^{\nu}}=\left\vert f\right\vert _{L_{p}}+\left\vert L^{\nu}f\right\vert _{L_{p}}.
\]

The notation $H_{p}^{\alpha}\left(\mathbf{R}^{d}\right)$ and $H_{p}^{\alpha}\left(E\right)$
are used when $\nu\left(dy\right)=dy/\left|y\right|^{d+\alpha},\alpha\in\left(0,2\right).$
They are consistent with classical Bessel potential spaces and in
this case $L^{\nu}=c\left(\alpha,d\right)\left(-\triangle\right)^{\alpha/2}$
where $\left(-\triangle\right)^{\alpha/2}$ is a fractional Laplacian. 

We will also work with another space of smooth functions $\tilde{C}^{\infty}\left(\mathbf{R}^{d}\right)$
and $\tilde{C}^{\infty}\left(E\right)$ the space of locally integrable
functions such that for all $1\leq p<\infty$ and multiindex $\gamma\in\mathbf{N}_{0}^{d}$,
\begin{equation}
\sup_{x}\left|D^{\gamma}f\left(x\right)\right|+\left|D^{\gamma}f\right|_{L_{p}}<\infty\label{eq:p0}
\end{equation}

\noindent and 
\begin{equation}
\int_{0}^{T}\sup_{x}\left|D^{\gamma}f\left(t,x\right)\right|dt+\left|D^{\gamma}f\right|_{L_{p}\left(E\right)}<\infty\label{eq:p1}
\end{equation}

\noindent respectively. For a fixed $1\leq p<\infty$, we use $\tilde{C}_{p}^{\infty}\left(\mathbf{R}^{d}\right)$
(resp. $\tilde{C}_{p}^{\infty}\left(E\right)$) the space of locally
integrable functions with finite norms (\ref{eq:p0}) and (\ref{eq:p1})
respectively. 

We now define the space in which we study (\ref{eq:mainEq}). Let
$\mathcal{H}_{p}^{\nu}\left(E\right)$ be the space of all functions
$u\in H_{p}^{\nu}\left(E\right)$ such that 
\[
u\left(t,x\right)=\int_{0}^{t}F\left(s,x\right)ds,0\leq t\leq T
\]

\noindent for some $F\in L_{p}\left(E\right).$ It is a Banach space
equipped with the norm 

\noindent 
\[
\left|u\right|_{\mathcal{H}_{p}^{\nu}\left(E\right)}=\left|u\right|_{H_{p}^{\nu}\left(E\right)}+\left|F\right|_{L_{p}\left(E\right)}.
\]

\begin{defn}
Let $f\in L_{p}\left(E\right)$. We say that $u\in\mathcal{H}_{p}^{\nu}\left(E\right)$
is a solution of (\ref{eq:mainEq}) if $L^{m,\nu}u\in L_{p}\left(E\right)$
and
\[
u\left(t\right)=\int_{0}^{t}\left[L^{m,\nu}u\left(s\right)-\lambda u\left(s\right)+f\left(s\right)\right]ds,\hspace{1em}0\leq t\leq T
\]
\end{defn}
\noindent holds in $L_{p}\left(\mathbf{R}^{d}\right).$

We are now ready to state our main results. 
\begin{thm}
\noindent \label{Thm:main} Let $\nu\in\mathfrak{A}^{\sigma},w=w_{\nu}$
be a continuous O-RV function and \textbf{A, B, C} hold. Let $\beta\in\left(0,1\right),p>d/\beta$.
Then for any $f\in L_{p}\left(E\right)$ there exists a unique solution
$u\in\mathcal{H}_{p}^{\nu}\left(E\right)$ to (\ref{eq:mainEq}).
Moreover, there exist $\lambda_{0}=\lambda_{0}\left(d,p,\nu,\beta,\kappa,k,K\right)>0$
and $N=N\left(d,p,\nu,\beta,\kappa,k,K\right)>0$ such that for any
$\lambda\geq\lambda_{0},$
\begin{align*}
\left|\partial_{t}u\right|_{L_{p}\left(E\right)}+\left\vert L^{\nu}u\right\vert _{L_{p}\left(E\right)} & \leq N\left\vert f\right\vert _{L_{p}\left(E\right)}\\
\left\vert u\right\vert _{L_{p}\left(E\right)} & \leq N\rho_{\lambda}\left\vert f\right\vert _{L_{p}\left(E\right)}
\end{align*}
where $\rho_{\lambda}=T\wedge\lambda^{-1}.$ In addition, there exists
$N=N\left(d,p,\nu,\beta,\kappa,k,K,T\right)>0$ such that for any
$\lambda\geq0$,
\begin{align*}
\left\vert u\right\vert _{H_{p}^{\nu}\left(E\right)} & \leq N\left\vert f\right\vert _{L_{p}\left(E\right)}.
\end{align*}
\end{thm}

\subsection{Examples}

In this section, we provide some examples to illustrate that the selected
assumptions can indeed cover a variety of interesting Lévy measures
that go beyond those comparable to $dy/\left|y\right|^{d+\alpha}.$
\begin{example}
\noindent \label{example:aniso} We start with a concrete example
of an anisotropic measure with the same order $\alpha\in\left(0,2\right)$
in every direction. Let $c_{i}>0,i=1,...,d$ and $\epsilon_{0}$ a
Dirac measure focused on the center in $\mathbf{R}^{d-1},$ 
\[
\nu\left(dy\right)=\sum_{i=1}^{d}c_{i}\left|y_{i}\right|^{-1-\alpha}dy_{i}\epsilon_{0}\left(dy_{1},...,dy_{i-1},dy_{i+1}dy_{d}\right).
\]

\noindent By direct computation,

\noindent 
\[
c_{i}\int_{\left|y\right|\geq r}\left|y_{i}\right|^{-1-\alpha}dy_{i}\epsilon_{0}\left(dy_{1},...,dy_{i-1},dy_{i+1}dy_{d}\right)=2c_{i}\int_{x\geq r}x^{-1-\alpha}dx=2c_{i}r^{-\alpha}.
\]

\noindent Thus, $\text{\ensuremath{\delta}}_{\nu}\left(r\right)=\nu\left(x\in\mathbf{R}^{d}:\left|x\right|>r\right)=2\left(\sum_{i=1}^{d}c_{i}\right)r^{-\alpha}$
and 
\[
w_{\nu}\left(r\right)=\frac{1}{2}\left(\sum_{i=1}^{d}c_{i}\right)^{-1}r^{\alpha}.
\]

\noindent We may also check assumption \textbf{B }straightforwardly,
for $\left|\hat{\xi}\right|=1$ and $r>0,$ 
\begin{align*}
 & c_{i}w\left(r\right)\int_{\left|y\right|\leq1}\left|y\cdot\hat{\xi}\right|^{2}\left|ry_{i}\right|^{-1-\alpha}d\left(ry_{i}\right)\epsilon_{0}\left(dy_{1},...,dy_{i-1},dy_{i+1}dy_{d}\right)\\
 & =2c_{i}w\left(r\right)r^{-\alpha}\int_{x\leq1}\left|\hat{\xi}_{i}\right|^{2}x^{1-\alpha}dx=\frac{c_{i}}{\left(2-\alpha\right)\left(\sum_{i=1}^{d}c_{i}\right)}\left|\hat{\xi}_{i}\right|^{2}.
\end{align*}

\noindent Therefore, for $\left|\hat{\xi}\right|=1$ and $r>0$,
\[
\int_{\left|y\right|\leq1}\left|y\cdot\hat{\xi}\right|^{2}\tilde{\nu}_{r}\left(dy\right)\geq\frac{\min_{i=1,...,d}\left\{ c_{i}\right\} }{\left(2-\alpha\right)\left(\sum_{i=1}^{d}c_{i}\right)}>0.
\]
\end{example}
\noindent The following two examples of measures are taken from \cite{MF}
with slight modifications. In fact, the above example is a special
case of the next example as well.
\begin{example}
\noindent According to \cite{DM} (pp. 70-74), any Lévy measure $\nu\in\mathfrak{A}^{\sigma}$
can be disintegrated as 
\[
\nu\left(\Gamma\right)=-\int_{0}^{\infty}\int_{S_{d-1}}\chi_{\Gamma}\left(rz\right)\Pi\left(r,dz\right)d\delta_{\nu}\left(r\right),\Gamma\in\mathcal{B}\left(\mathbf{R}_{0}^{d}\right),
\]
where $\delta=\delta_{\nu}$, and $\Pi\left(r,d\omega\right),r>0$
is a measurable family of measures on the unit sphere $S_{d-1}$ with
$\Pi\left(r,S_{d-1}\right)=1,r>0.$ If 

\noindent (i) $w_{\nu}=\delta^{-1}$ is an continuous, O-RV and satisfies
Assumption $\mathbf{A}$ 

\noindent (ii) $\left\vert \left\{ s\in\left[0,1\right]:r_{i}\left(s\right)<1\right\} \right\vert >0,i=1,2,$
and 

\noindent (iii) $\inf_{\left\vert \hat{\xi}\right\vert =1}\int_{S_{d-1}}\left\vert \hat{\xi}\cdot z\right\vert ^{2}\Pi\left(r,dz\right)\geq c_{0}>0,r>0,$ 

\noindent then all assumptions of Theorem \ref{Thm:main} holds.
\end{example}
\noindent 
\begin{example}
\noindent Consider Lévy measure in radial and angular coordinate in
the form 
\[
\nu\left(B\right)=\int_{0}^{\infty}\int_{\left\vert \omega\right\vert =1}1_{B}\left(rz\right)a\left(r,z\right)j\left(r\right)r^{d-1}S\left(dz\right)dr,B\in\mathcal{B}\left(\mathbf{R}_{0}^{d}\right),
\]
where $S\left(dz\right)$ is a finite measure on the unit sphere.

\noindent Assume

\noindent (i) There is $C>1,c>0,0<\delta_{1}\leq\delta_{2}<1$ such
that 
\[
C^{-1}\phi\left(r^{-2}\right)\leq j\left(r\right)r^{d}\leq C\phi\left(r^{-2}\right)
\]
and for all $0<r\leq R$, 
\[
c^{-1}\left(\frac{R}{r}\right)^{\delta_{1}}\leq\frac{\phi\left(R\right)}{\phi\left(r\right)}\leq c\left(\frac{R}{r}\right)^{\delta_{2}}.
\]

\noindent (ii) There is a function $\rho_{0}\left(\omega\right)$
defined on the unit sphere such that $\rho_{0}\left(z\right)\leq a\left(r,z\right)\leq1,r>0,z\in S_{d-1}$,
and for all $\left\vert \hat{\xi}\right\vert =1$, 
\[
\int_{S^{d-1}}\left\vert \hat{\xi}\cdot z\right\vert ^{2}\rho_{0}\left(z\right)S\left(dz\right)\geq c>0.
\]
Under these assumptions it can be shown that $\mathbf{B}$ holds,
and $w_{\nu}$ is an O-RV function with $2\delta_{1}\leq p_{i}^{\nu}\leq q_{i}^{\nu}\leq2\delta_{2},i=1,2.$
To provide some examples of $\phi$ satisfying (i), we can take $\phi\left(r\right)=r^{\alpha}\left(\ln\left(1+r\right)\right)^{\beta},\alpha\in\left(0,1\right),\beta\in\left(0,1-\alpha\right)$
or $\phi\left(r\right)=\left[\ln\left(\cosh\sqrt{r}\right)\right]^{\alpha},$
$\alpha\in\left(0,1\right).$ 
\end{example}

\section{Symbol and Probability Density Estimates}

\noindent Let $\nu\in\mathfrak{A}^{\sigma},w=w_{\nu}$ be a continuous
O-RV function and \textbf{A, B} hold, $0<k\leq m\left(t,y\right)\leq K,t\in\left[0,T\right],y\in\mathbf{R}^{d}$
and $p^{m,\nu}\left(dt,dy\right)$ be a Poisson point measure on $\left[0,\infty\right)\times\mathbf{R}_{0}^{d}$
such that $\mathbf{E}p^{m,\nu}\left(dt,dy\right)=m\left(t,y\right)\nu\left(dy\right)dt.$ 

Let $q^{m,\nu}\left(dt,dy\right)=q^{m,\nu}\left(dt,dy\right)=p^{m,\nu}\left(dt,dy\right)-m\left(t,y\right)\nu\left(dy\right)dt$
be its compensator. We associate to $L^{m,\nu}$ the stochastic process
with independent increments 
\begin{equation}
Z_{t}^{m,\nu}=\int_{0}^{t}\int\chi_{\sigma}\left(y\right)yq^{m,\nu}\left(ds,dy\right)+\int_{0}^{t}\int\left(1-\chi_{\sigma}\left(y\right)\right)yp^{m,\nu}\left(ds,dy\right),t\in\left[0,T\right].\label{f10}
\end{equation}
By It\^{o} formula (e.g. \cite[Proposition 1]{Mik1} or \cite[Proposition 8.19]{CT}),
\[
\mathbf{E}e^{i2\pi\xi\cdot\left(Z_{t}^{m,\nu}-Z_{s}^{m,\nu}\right)}=\exp\left\{ \int_{s}^{t}\psi^{m,\nu}\left(r,\xi\right)dr\right\} ,0<s<t\leq T,\xi\in\mathbf{R}^{d},
\]
where the symbol $\psi^{m,\nu}$ is given by 
\[
\psi^{m,\nu}\left(r,\xi\right)=\int\left[e^{i2\pi\xi\cdot y}-1-i2\pi y\cdot\xi\chi_{\sigma}\left(y\right)\right]m\left(r,y\right)\nu\left(dy\right).
\]

We will show in Lemma \ref{lem:timeDensity} below that 
\begin{equation}
p^{m,\nu}\left(s,t\right)=\mathcal{F}\exp\left\{ \int_{s}^{t}\psi^{m,\nu}\left(r,\xi\right)dr\right\} \label{eq:p_Fourier}
\end{equation}
is the probability density of $Z_{t}^{m,\nu}-Z_{s}^{m,\nu}$. Note
that in the case of $m=1$, $Z_{t}^{\nu}=Z_{t}^{1,\nu}$ has a stationary
independent increment and we write $p^{1,\nu}\left(s,t\right)=p^{\nu}\left(t-s\right)$
and $\psi^{1,\nu}\left(r,\xi\right)=\psi^{\nu}\left(\xi\right)$ respectively. 

The following estimates of the symbol without a coefficient is an
immediate consequence of \cite[Lemma 7]{MPh1}. 
\begin{lem}
\noindent \label{lem:symbolEst}Let $\nu\in\mathfrak{A}^{\sigma},w=w_{\nu}$
be a continuous O-RV function and\textbf{ A, B} hold. Let $\pi\prec\nu\left(M,\alpha_{1},\alpha_{2}\right)$
and $\tilde{\pi}_{R}\left(dy\right)=w\left(R\right)\pi\left(Rdy\right),R>0$
then there is $N=N\left(\nu\right)>0$ so that 
\[
\left|\psi^{\tilde{\pi}_{R}}\left(\xi\right)\right|\leq NMw\left(\left|\xi\right|^{-1}\right)^{-1},\hspace{1em}\left|\xi\right|\in\mathbf{R}^{d},
\]

\noindent assuming $w\left(\left|\xi\right|^{-1}\right)^{-1}=0$ if
$\xi=0$. Moreover, if $\pi=m\left(y\right)\nu\left(dy\right)$ where
$0<k\leq m\left(y\right)\leq K,y\in\mathbf{R}^{d}$ then there are
$N=N\left(\nu\right)$ and $n=n\left(\nu\right)>0$ so that

\noindent 
\[
\left|\psi^{\tilde{\pi}_{R}}\left(\xi\right)\right|\leq NKw\left(\left|\xi\right|^{-1}\right)^{-1},\hspace{1em}\left|\xi\right|\in\mathbf{R}^{d},
\]
\[
\mathfrak{R}\left\{ \psi^{\tilde{\pi}_{R}}\left(\xi\right)\right\} \leq-nkw\left(\left|\xi\right|^{-1}\right)^{-1},\hspace{1em}\left|\xi\right|\in\mathbf{R}^{d}.
\]
\end{lem}
Some probability estimates for the case $m=1$ were derived in \cite[Lemma 8]{MPh3}.
We will extend those estimates to $p^{m,\nu}\left(s,t\right)$. Indeed,
changing the variable of integration, denoting $R=a\left(t-s\right)$
where 
\[
a\left(r\right)=\inf\left\{ t>0:w_{\nu}\left(t\right)\geq r\right\} ,r>0
\]
is the left continuous inverse of $w_{\nu}$, then at least formally
for $0<s<t\leq T$ and $x\in\mathbf{R}^{d},$
\begin{align*}
 & p^{m,\nu}\left(s,t,x\right)\\
 & =R^{-d}\mathcal{F}\exp\left\{ \int_{s}^{t}\int\left[e^{i2\pi\xi\cdot y}-1-i2\pi y\cdot\xi\chi_{\sigma}\left(y\right)\right]m\left(r,Ry\right)\nu\left(Rdy\right)dr\right\} \left(\frac{x}{R}\right)\\
 & =R^{-d}\mathcal{F}\exp\left\{ \left(t-s\right)^{-1}\int_{s}^{t}\int\left[e^{i2\pi\xi\cdot y}-1-i2\pi y\cdot\xi\chi_{\sigma}\left(y\right)\right]m_{R}\left(r,y\right)\tilde{\nu}_{R}\left(dy\right)dr\right\} \left(\frac{x}{R}\right)\\
 & =R^{-d}\mathcal{F}\exp\left\{ \int\left[e^{i2\pi\xi\cdot y}-1-i2\pi y\cdot\xi\chi_{\sigma}\left(y\right)\right]\bar{m}_{R}^{s,t}\left(y\right)\tilde{\nu}_{R}\left(dy\right)\right\} \left(\frac{x}{R}\right).
\end{align*}

\noindent where $\bar{m}_{R}^{s,t}\left(y\right)=\text{\ensuremath{\left(t-s\right)}}^{-1}\int_{s}^{t}m_{R}\left(r,y\right)dr,$
clearly $k\leq\bar{m}_{R}^{s,t}\left(y\right)\leq K,y\in\mathbf{R}^{d}.$ 

\medskip

That is, recalling $R=a\left(t-s\right)$, 

\noindent 
\begin{align}
p^{m,\nu}\left(s,t,x\right) & =R^{-d}\bar{p}^{m_{a\left(t-s\right)},\tilde{\nu}_{a\left(t-s\right)}}\left(s,t,\frac{x}{a\left(t-s\right)}\right),0<s<t\leq T,x\in\mathbf{R}^{d}\label{eq:scale_p0}
\end{align}

\noindent where for $R>0,$
\begin{align}
 & \bar{p}^{m_{R},\tilde{\nu}_{R}}\left(s,t,x\right)\label{eq:pbar}\\
= & \mathcal{F}\left\{ \exp\left(\int\left[e^{i2\pi\xi\cdot y}-1-i2\pi y\cdot\xi\chi_{\sigma}\left(y\right)\right]\bar{m}_{R}^{s,t}\left(y\right)\tilde{\nu}_{R}\left(dy\right)\right)\right\} \left(x\right)\nonumber \\
= & \mathcal{\mathcal{F}}\exp\left\{ \psi^{\bar{m}_{R}^{s,t},\tilde{\nu}_{R}}\left(\xi\right)\right\} \left(x\right).\nonumber 
\end{align}

We may also express (\ref{eq:scale_p0}) in terms of symbols, 

\begin{equation}
\exp\left\{ \int_{s}^{t}\psi^{m,\nu}\left(r,\xi\right)dr\right\} =\exp\left\{ \psi^{\bar{m}_{a\left(t-s\right)}^{s,t},\tilde{\nu}_{a\left(t-s\right)}}\left(a\left(t-s\right)\xi\right)\right\} .\label{eq:scale_symb}
\end{equation}
With the above scaling computation in mind, we derive estimates of
$\bar{p}^{m_{R},\tilde{\nu}_{R}}.$ 
\begin{lem}
\noindent \label{lem:timeDensity} Let $\nu\in\mathfrak{A}^{\sigma},w=w_{\nu}$
be a continuous O-RV function and \textbf{A, B} hold. Let $0<k\leq m\left(t,y\right)\leq K,t\in\left[0,T\right],y\in\mathbf{R}^{d},$$\pi\prec\nu\left(M,\alpha_{1},\alpha_{2}\right)$
and $\tilde{\pi}_{R}\left(dy\right)=w\left(R\right)\pi\left(Rdy\right),R>0.$
Then the following statements hold.

\bigskip

\noindent (i) $\bar{p}^{m_{R},\tilde{\nu}_{R}}$ defined by (\ref{eq:pbar})
is a probability density and $p^{m,\nu}\left(s,t\right)$ defined
by (\ref{eq:p_Fourier}) is the probability density of $Z_{t}^{m,\nu}-Z_{s}^{m,\nu},0<s<t\leq T$. 

\bigskip

\noindent (ii) For any multiindex $\eta\in\mathbf{N}_{0}^{d}$, there
is $N=N\left(\eta,\alpha_{2},\nu,k,K\right)>0$ such that
\[
\int\left(1+\left|x\right|^{\alpha_{2}}\right)\left|D^{\eta}L^{\tilde{\pi}_{R}}\bar{p}^{m_{R},\tilde{\nu}_{R}}\left(s,t,x\right)\right|dx\leq NM,\hspace{1em}0<s<t\leq T,R>0.
\]

\medskip

\noindent (iii) If $\left|m^{\prime}\left(y\right)\right|\leq K_{1},y\in\mathbf{R}^{d}$,
then there is $N=N\left(\alpha_{2},\nu,k,K,K_{1}\right)>0$ such that 

\noindent 
\[
\int\left(1+\left|x\right|^{\alpha_{2}}\right)\left|L^{\tilde{\pi}_{R}}L^{m^{\prime},\tilde{\nu}_{R}^{\ast}}\bar{p}^{m_{R},\tilde{\nu}_{R}}\left(s,t,x\right)\right|dx\leq NM,\hspace{1em}0<s<t\leq T,R>0.
\]
\end{lem}
\begin{proof}
\noindent We only provide a sketch of the proof. The key idea is that
the integral over the region $\left|y\right|\leq1$ in (\ref{eq:pbar})
provides exponential decay of the symbol. Indeed, for $\left|\xi\right|\geq1,\hat{\xi}=\xi/\left|\xi\right|,$

\begin{align*}
 & \int_{\left|y\right|\leq1}\left[e^{i2\pi\left|\xi\right|\hat{\xi}\cdot y}-1-i2\pi\left|\xi\right|y\cdot\hat{\xi}\chi_{\sigma}\left(y\right)\right]\bar{m}_{R}^{s,t}\left(y\right)\tilde{\nu}_{R}\left(dy\right)\\
 & =\frac{w\left(R\right)}{w\left(\left|\xi\right|^{-1}R\right)}\int_{\left|y\right|\leq\left|\xi\right|}\left[e^{i2\pi\hat{\xi}\cdot y}-1-i2\pi y\cdot\hat{\xi}\chi_{\sigma}\left(y\right)\right]\bar{m}_{R}^{s,t}\left(\left|\xi\right|^{-1}y\right)\tilde{\nu}_{\left|\xi\right|^{-1}R}\left(dy\right)
\end{align*}

Therefore, by Lemma \ref{lem:alpha12_ratio}\textbf{ }and applying
a trigonometric inequality $1-\cos\left(x\right)\geq x^{2}/\pi$ for
$\left|x\right|\leq\pi/2$, we obtain 

\noindent 
\begin{align*}
 & \mathfrak{R}\left\{ \frac{w\left(R\right)}{w\left(\left|\xi\right|^{-1}R\right)}\int_{\left|y\right|\leq\left|\xi\right|}\left[e^{i2\pi\hat{\xi}\cdot y}-1-i2\pi y\cdot\hat{\xi}\chi_{\sigma}\left(y\right)\right]\bar{m}_{R}^{s,t}\left(\left|\xi\right|^{-1}y\right)\tilde{\nu}_{\left|\xi\right|R}\left(dy\right)\right\} \\
\leq & ck\left|\xi\right|^{\alpha_{2}}\int_{\left|y\right|\leq\left|\xi\right|}\left[\text{\ensuremath{\cos}}\left(2\pi\hat{\xi}\cdot y\right)-1\right]\tilde{\nu}_{\left|\xi\right|R}\left(dy\right)\\
\leq & -c\left|\xi\right|^{\alpha_{2}}\int_{\left|y\right|\leq1/4}\left|\hat{\xi}\cdot y\right|^{2}\tilde{\nu}_{\left|\xi\right|R}\left(dy\right)\leq-c\left|\xi\right|^{\alpha_{2}}.
\end{align*}

The above estimate and Proposition I.2.5 in \cite{Sat} ensures that
$\bar{p}{}^{m_{R},\tilde{\nu}_{R}}\left(s,t\right)$ and hence $p^{m,\nu}\left(s,t\right)$
are probability densities. The proof of estimates are basically a
repetition of \cite[Lemmas 6-8]{MPh3} with $\bar{m}_{R}^{s,t}\left(y\right)\tilde{\nu}_{R}\left(dy\right)$
instead of $\tilde{\nu}_{R}\left(dy\right)$ (see also references
therein.) The presence of bounded $\bar{m}_{R}^{s,t}\left(y\right)$
does not cause any additional difficulties. 
\end{proof}
We now investigate further some preparatory estimates. To ease notation,
we will write $\bar{p}^{R}\left(s,t\right)=\bar{p}^{m_{R},\tilde{\nu}_{R}}\left(s,t\right)$.
These fundamental estimates are essential for verifying Hörmander
condition. 

Recall that under the assumptions of Lemma \ref{lem:timeDensity},
we have from (\ref{eq:scale_symb})
\[
\exp\left\{ \int_{s}^{t}\psi^{m,\nu}\left(r,\xi\right)dr\right\} =\exp\left\{ \psi^{\bar{m}_{a\left(t-s\right)}^{s,t},\tilde{\nu}_{a\left(t-s\right)}}\left(a\left(t-s\right)\xi\right)\right\} .
\]

\noindent For $\pi,\mu\in\mathfrak{A}^{\sigma}$,

\begin{align*}
 & \psi^{\pi}\exp\left\{ \int_{s}^{t}\psi^{m,\nu}\left(r,\xi\right)dr\right\} \\
 & =\frac{1}{t-s}\psi^{\tilde{\pi}_{a\left(t-s\right)}}\left(a\left(t-s\right)\xi\right)\exp\left\{ \psi^{\bar{m}_{a\left(t-s\right)}^{s,t},\tilde{\nu}_{a\left(t-s\right)}}\left(a\left(t-s\right)\xi\right)\right\} 
\end{align*}

We deduce by taking the Fourier inverse that 
\begin{equation}
L^{\pi}p^{m,\nu}\left(s,t,x\right)=\frac{1}{t-s}a\left(t-s\right)^{-d}\left(L^{\tilde{\pi}_{a\left(t-s\right)}}\bar{p}^{a\left(t-s\right)}\right)\left(s,t,\frac{x}{a\left(t-s\right)}\right).\label{eq:LpiP}
\end{equation}

\noindent Similarly, 
\begin{equation}
L^{\pi}L^{\mu}p^{m,\nu}\left(s,t,x\right)=\frac{1}{\left(t-s\right)^{2}}a\left(t-s\right)^{-d}\left(L^{\tilde{\pi}_{a\left(t-s\right)}}L^{\tilde{\mu}_{a\left(t-s\right)}}L\bar{p}^{a\left(t-s\right)}\right)\left(s,t,\frac{x}{a\left(t-s\right)}\right).\label{eq:LLp}
\end{equation}

\begin{lem}
\noindent \label{lem:prelim_hormander} Let $\nu\in\mathfrak{A}^{\sigma},w=w_{\nu}$
be a continuous O-RV function and \textbf{A, B} hold. Let $0<k\leq m\left(t,y\right)\leq K,t\in\left[0,T\right],y\in\mathbf{R}^{d}$
and $\pi\prec\nu\left(M,\alpha_{1},\alpha_{2}\right).$ Then 

\medskip

\noindent (i) For any $\beta\in\left[0,\alpha_{2}\right]$, there
exists $N=N\left(\beta,\nu,k,K,M\right)>0$ such that for all $0<s<t\leq T$
and $c>0$,
\[
\int_{s}^{t}\int_{\left|x\right|>c}\left|L^{\pi}p^{m,\nu}\left(r,t,x\right)\right|dxdr\leq Nc^{-\beta}a\left(t-s\right)^{\beta}.
\]

\noindent (ii) For $0<b<t\leq T$, $h\in\mathbf{R}^{d},$ there exists
$N=N\left(\nu,k,K,M\right)>0$ such that 
\[
\int_{0}^{b}\int\left|L^{\pi}p^{m,\nu}\left(r,t,x+h\right)-L^{\pi}p^{m,\nu}\left(r,t,x\right)\right|dxdr\leq N\left|h\right|a\left(t-b\right)^{-1}.
\]

\noindent (iii) For $0<b<s<t\leq T$, there exists $N=N\left(\nu,k,K,M\right)>0$
such that
\[
\int_{0}^{b}\int\left|L^{\pi}p^{m,\nu}\left(r,t,x\right)-L^{\pi}p^{m,\nu}\left(r,s,x\right)\right|dxdr\leq N\left(t-s\right)\left(s-b\right)^{-1}.
\]
\end{lem}
\begin{proof}
We will apply Lemma \ref{lem:timeDensity} along with (\ref{eq:LpiP})
and (\ref{eq:LLp}) repeatedly. 

\noindent (i) By changing the variable of integration, Chebyshev inequality
and Lemma \ref{lem:timeDensity} (ii),
\begin{align*}
\int_{s}^{t}\int_{\left|x\right|>c}\left|L^{\pi}p^{m,\nu}\left(r,t,x\right)\right|dxdr & =\int_{s}^{t}\frac{a\left(t-r\right)^{-d}}{t-r}\int_{\left|x\right|>c}\left|L^{\tilde{\pi}_{a\left(t-r\right)}}\bar{p}^{a\left(t-r\right)}\left(r,t,\frac{x}{a\left(t-r\right)}\right)\right|dxdr\\
 & =\int_{s}^{t}\frac{1}{t-r}\int_{\left|x\right|>ca\left(t-r\right)^{-1}}\left|L^{\tilde{\pi}_{a\left(t-r\right)}}\bar{p}^{a\left(t-r\right)}\left(r,t,x\right)\right|dxdr\\
 & \leq N\int_{s}^{t}\frac{c^{-\beta}a\left(t-r\right)^{\beta}}{t-r}dr=Nc^{-\beta}\int_{0}^{t-s}\frac{a\left(r\right)^{\beta}}{r}dr\\
 & \leq Nc^{-\beta}a\left(t-s\right)^{\beta}
\end{align*}

\noindent where we apply Corollary \ref{cor:aymp_a} in the last step.

\noindent (ii) By the fundamental theorem of calculus, changing the
variable of integration and Lemma \ref{lem:timeDensity} (ii),
\begin{align*}
 & \int_{0}^{b}\int\left|L^{\pi}p^{m,\nu}\left(r,t,x+h\right)-L^{\pi}p^{m,\nu}\left(r,t,x\right)\right|dxdr\\
\leq & \left|h\right|\int_{0}^{b}\int\int_{0}^{1}\left|L^{\pi}\nabla p^{m,\nu}\left(r,t,x+lh\right)\right|dldxdr\\
\leq & N\left|h\right|\int_{0}^{b}a\left(t-r\right)^{-1}\left(t-r\right)^{-1}dr\\
\leq & N\left|h\right|\int_{t-b}^{t}a\left(r\right)^{-1}r^{-1}dr\\
\leq & N\left|h\right|a\left(t-b\right)^{-1}
\end{align*}

\noindent where we apply Corollary \ref{cor:aymp_a} in the last inequality.

\noindent (iii) By an elementary computation, denoting $m^{\ast}\left(t,y\right)=m\left(t,-y\right)$,
\[
\partial_{t}p^{m,\nu}\left(s,t\right)=L^{m^{\ast}\left(t\right),\nu^{\ast}}p^{m,\nu}\left(s,t\right).
\]
By the fundamental theorem of calculus, changing the variable of integration,
and Lemma \ref{lem:timeDensity} (iii), 
\begin{align*}
 & \int_{0}^{b}\int\left|L^{\pi}p^{m,\nu}\left(r,t,x\right)-L^{\pi}p^{m,\nu}\left(r,s,x\right)\right|dxdr\\
\leq & \left(t-s\right)\int_{0}^{b}\int\int_{0}^{1}L^{\pi}L^{m^{\ast}\left(lt+\left(1-l\right)s\right),\nu^{\ast}}p^{m,\nu}\left(r,lt+\left(1-l\right)s,x\right)dldxdr\\
\leq & N\left(t-s\right)\int_{0}^{b}\int_{0}^{1}\frac{1}{\left(lt+\left(1-l\right)s-r\right)^{2}}dldr\\
\leq & N\left(t-s\right)\left(s-b\right)^{-1}.
\end{align*}

The proof is complete.
\end{proof}
\noindent 
For an isotropic unimodal measure without coefficient, we may derive
finer estimates (albeit as a cost of generality.) We define for $a_{1},b_{1}>0,$$d\geq1,$and
$\gamma$ in Assumption \textbf{D}, 

\noindent 
\[
\eta_{a_{1},b_{1}}^{d}\left(t,x\right)=\left[t\frac{\exp\left(-b_{1}t\gamma\left(\left|x\right|\right)^{-1}\right)}{\left|x\right|^{d}\gamma\left(\left|x\right|\right)}\chi_{t\leq a_{1}\gamma\left(\left|x\right|\right)}+\left[a_{\gamma}\left(a_{1}/t\right)\right]^{d}\chi_{t\geq a_{1}\gamma\left(\left|x\right|\right)}\right],\left(t,x\right)\in E.
\]

\begin{lem}
\noindent \label{lem:gradient} Let $\nu\left(dy\right)=j_{d}\left(\left|y\right|\right)dy\in\mathfrak{A}^{\sigma}$
and\textbf{ D} hold. Then there are $a_{i},b_{i},c_{i}>0,i=1,2,3$
such that for all $t\in\left[0,T\right],x\in\mathbf{R}_{0}^{d},$

\medskip

\noindent (i) $\left|p^{\nu}\left(t,x\right)\right|\leq c_{1}\eta_{a_{1},b_{1}}^{d}\left(t,x\right),$

\noindent (ii) $\left|D_{x}p^{\nu}\left(t,x\right)\right|\leq c_{2}\left|x\right|\eta_{a_{2},b_{2}}^{d+2}\left(t,x\right).$
\end{lem}
\begin{proof}
\noindent (i) The first estimate follows from \cite[Proposition 2.9]{CKK}
whose assumptions can be easily verified due to our assumption \textbf{D}
(for $n=0$.) Indeed, there are $a_{1},b_{1},c_{1}>0$ such that
\begin{equation}
\left|p^{\nu}\left(t,x\right)\right|\leq c_{1}t\frac{K\left(\theta_{a_{1}}\left(\left|x\right|,t\right)\right)}{\theta_{a_{1}}\left(\left|x\right|,t\right)^{d}}\exp\left(-b_{1}th\left(\theta_{a_{1}}\left(\left|x\right|,t\right)\right)\right),\left(t,x\right)\in E.\label{eq:prob0}
\end{equation}

\noindent where $\theta_{a_{1}}\left(r,t\right)=r\vee\left[a_{\gamma}\left(a_{1}/t\right)\right]^{-1}$
and $a_{\gamma}$ is a right continuous inverse of $\bar{\gamma}\left(r\right)=\gamma\left(r^{-1}\right)^{-1},r>0$
and
\[
h\left(r\right)=K\left(r\right)+L\left(r\right)=r^{-2}\int_{0}^{r}s\gamma\left(s\right)^{-1}ds+\int_{r}^{\infty}s^{-1}\gamma\left(s\right)^{-1}ds,\hspace{1em}r>0.
\]

By Lemmas \ref{lem:al1}, \ref{lem:al2}, $K\left(r\right)\preceq\gamma\left(r\right)^{-1},L\left(r\right)\preceq\gamma\left(r\right)^{-1}$
and by Lemma \ref{lem:w=00003Dgamma}, $h\left(r\right)\succeq L\left(r\right)\succeq\gamma\left(r\right)^{-1}$.
Consequently, (\ref{eq:prob0}) can be simplified to 

\noindent 
\[
\left|p^{\nu}\left(t,x\right)\right|\leq c_{1}t\left|\theta_{a_{1}}\left(\left|x\right|,t\right)\right|^{-d}\gamma\left(\theta_{a_{1}}\left(\left|x\right|,t\right)\right)^{-1}\exp\left(-b_{1}t\gamma\left(\theta_{a_{1}}\left(\left|x\right|,t\right)\right)^{-1}\right).
\]

We note that since $\bar{\gamma}$ is an non-decreasing continuous
function, $\bar{\gamma}\left(a_{\gamma}\left(r\right)\right)=r.$

\noindent If $\left|x\right|\geq\left[a_{\gamma}\left(a_{1}/t\right)\right]^{-1}$
then $t\leq a_{1}\gamma\left(\left|x\right|\right)$, $\theta_{a_{1}}\left(\left|x\right|,t\right)=\left|x\right|$
and thus
\[
\left|p^{\nu}\left(t,x\right)\right|\leq c_{1}t\left|x\right|^{-d}\gamma\left(\left|x\right|\right)^{-1}\exp\left(-b_{1}t\gamma\left(\left|x\right|\right)^{-1}\right).
\]

\noindent If $\left|x\right|<\left[a_{\gamma}\left(a_{1}/t\right)\right]^{-1}$
then $t\geq a_{1}\gamma\left(\left|x\right|\right)$, $\gamma\left(\theta_{a_{1}}\left(\left|x\right|,t\right)\right)^{-1}=\gamma\left(\left[a_{\gamma}\left(a_{1}/t\right)\right]^{-1}\right)^{-1}=a_{1}/t$
and thus
\begin{align*}
\left|p^{\nu}\left(t,x\right)\right| & \leq c_{1}\left[a_{\gamma}\left(a_{1}/t\right)\right]^{d}.
\end{align*}

\noindent (ii) Now for the gradient, according to the proof of \cite[Theorem 1.5]{KR}
(see also \cite[Remark 2.3]{KP}), there exists a Lévy process $Z_{t}^{\left(d+2\right)}$
in $\mathbf{R}^{d+2}$ with the characteristic exponent $\psi^{\left(d+2\right)}\left(\xi\right)=\psi^{\nu_{d}}\left(\left|\xi\right|\right),\xi\in\mathbf{R}^{d+2}$
and whose radial, radially non-increasing transitional probability
satisfies

\noindent 
\[
p^{\left(d+2\right)}\left(t,r\right)=-\frac{1}{2\pi r}\frac{d}{dr}p^{\nu}\left(t,r\right),\hspace{1em}r>0.
\]

Moreover, the kernel is given by $j_{d+2}\left(r\right)=-\frac{1}{2\pi r}\frac{d}{dr}j_{d}\left(r\right),r>0$
and Assumption \textbf{D }(ii) for the gradient of $j_{d}$ becomes
\[
c_{1}r^{-d-2}\gamma\left(r\right)^{-1}\leq j_{d+2}\left(r\right)\leq c_{2}r^{-d-2}\gamma\left(r\right)^{-1},r>0.
\]

Hence, applying (i) for $p^{\left(d+2\right)}$ we obtain the upper
bound for the gradient of $p^{\nu}.$ 
\end{proof}
\noindent 
\begin{cor}
\noindent \label{cor:diff_D} Let $\nu\left(dy\right)=j_{d}\left(\left|y\right|\right)dy\in\mathfrak{A}^{\sigma}$
and\textbf{ D} hold. Then there exists $N>0$ such that 

\noindent (i) if $1<\frac{d}{q_{1}}\wedge\frac{d}{q_{2}}$ , then
$\int_{0}^{\infty}\left|p^{\nu}\left(t,a\right)\right|dt\leq N\frac{\gamma\left(\left|a\right|\right)}{\left|a\right|^{d}},a\in\mathbf{R}_{0}^{d}.$ 

\noindent (ii) $\int_{0}^{\infty}\left|p^{\nu}\left(t,a+b\right)-p^{\nu}\left(t,a\right)\right|dt\leq Nb\int_{0}^{1}\frac{\gamma\left(\left|a+lb\right|\right)}{\left|a+lb\right|^{d+1}}dl,$
for $a,b\in\mathbf{R}^{d}$ satisfying $\left|a+lb\right|\neq0$ for
all $l\in\left[0,1\right].$
\end{cor}
\begin{proof}
\noindent We start with the proof for (ii). By the fundamental theorem
of calculus and Lemma \ref{lem:gradient},

\noindent 
\begin{align*}
 & \int_{0}^{\infty}\left|p^{\nu}\left(t,a+b\right)-p^{\nu}\left(t,a\right)\right|dt\\
 & \leq b\int_{0}^{1}\int_{0}^{\infty}\left|\nabla p^{\nu}\left(t,a+lb\right)\right|dtdl\\
 & \leq b\left[\int_{0}^{1}\int_{0}^{a_{1}\gamma\left(\left|a+lb\right|\right)}t\frac{\exp\left(-b_{1}t\gamma\left(\left|a+lb\right|\right)^{-1}\right)}{\left|a+lb\right|^{d+1}\gamma\left(\left|a+lb\right|\right)}dtdl\right.\\
 & +\left.\int_{0}^{1}\int_{a_{1}\gamma\left(\left|a+lb\right|\right)}^{\infty}\left|a+lb\right|\left[a_{\gamma}\left(a_{1}t^{-1}\right)\right]^{d+2}dtdl\right]\\
 & =b\left(\mathcal{I}_{1}+\mathcal{I}_{2}\right).
\end{align*}

\noindent By direct integration, $\mathcal{I}_{1}\leq N\int_{0}^{1}\frac{\gamma\left(\left|a+lb\right|\right)}{\left|a+lb\right|^{d+1}}dl.$
Applying Corollary \ref{cor:aymp_a} and Lemma \ref{lem:index_a}
(iii) for $a_{\gamma}$, 
\begin{align*}
\mathcal{I}_{2} & =\int_{0}^{1}\int_{a_{1}\gamma\left(a+lb\right)}^{\infty}\left|a+lb\right|\left[a_{\gamma}\left(a_{1}t^{-1}\right)\right]^{d+2}dtdl\\
 & \leq N\int_{0}^{1}\left|a+lb\right|\left[a_{\gamma}\left(\gamma\left(\left|a+lb\right|\right)^{-1}\right)\right]^{d+2}\gamma\left(\left|a+lb\right|\right)dl\\
 & \leq N\int_{0}^{1}\frac{\gamma\left(\left|a+lb\right|\right)}{\left|a+lb\right|^{d+1}}dl.
\end{align*}

We derive (i) similarly. The restriction $1<\frac{d}{q_{1}}\wedge\frac{d}{q_{2}}$
comes from Corollary \ref{cor:aymp_a}. We note that if $\sigma\in\left(0,1\right)$,
then we always have $1<\frac{d}{q_{1}}\wedge\frac{d}{q_{2}}$ due
to Assumption \textbf{A}.
\end{proof}
The following representation of difference can be proved by following
closely the argument in \cite[Lemma 8]{MPh2} where the result was
derived under different assumptions. We state the result without proof.
\begin{lem}
\noindent \label{lem:diff_frac} Let $\delta\in\left(0,2\right)$,
$\nu\in\mathfrak{A}^{\sigma},w=w_{\nu}$ is a continuous O-RV function
and \textbf{A, B} hold. 

\noindent Define 
\[
\psi^{\nu;\delta}=\begin{cases}
\psi^{\nu} & \text{if}\hspace{1em}\delta=1\\
-\left(-\psi^{\nu_{sym}}\right)^{\delta} & \text{if}\hspace{1em}\delta\in\left(0,2\right),\delta\neq1,
\end{cases}
\]

\noindent and 
\[
L_{\delta}^{\nu}u=\mathcal{F}^{-1}\left[\psi^{\nu;\delta}\mathcal{F}u\right],u\in\mathcal{S}\left(\mathbf{R}^{d}\right).
\]

\noindent Assume that $0<\delta<\frac{1}{q_{1}}\wedge\frac{1}{q_{2}}$
then there is a constant $c>0$ so that for any $u\in\mathcal{S}\left(\mathbf{R}^{d}\right)$,
\[
u\left(x+z\right)-u\left(x\right)=c\int k^{\nu;\delta}\left(y,z\right)L_{\delta}^{\nu}u\left(x-y\right)dy,x\in\mathbf{R}^{d},z\in\mathbf{R}_{0}^{d},
\]

\noindent where the constant $c=c\left(\delta,d,\nu\right)$, $k^{\nu;1}=\bar{k}^{\nu;1},k^{\nu;\delta}=\bar{k}^{\nu_{sym};\delta}$
if $\delta\neq1$, and $\bar{k}^{\nu;\delta}$ (and respectively $\bar{k}^{\nu_{sym};\delta}$)
is defined by 
\[
\bar{k}^{\nu;\delta}\left(y,z\right)=w\left(\left|z\right|\right)^{\delta}\left|z\right|^{-d}\int_{0}^{\infty}t^{\delta}\left[p^{\tilde{\nu}_{\left|z\right|}}\left(t,\frac{y}{\left|z\right|}+\hat{z}\right)-p^{\tilde{\nu}_{\left|z\right|}}\left(t,\frac{y}{\left|z\right|}\right)\right]\frac{dt}{t}.
\]

\noindent Moreover, 
\[
\int\left|k^{\nu;\delta}\left(y,z\right)\right|dy\leq w\left(\left|z\right|\right)^{\delta}.
\]
\end{lem}
\begin{rem}
Note that in the above Lemma, $L_{\delta}^{\nu}=L^{\nu;\delta}$ if
$\delta\neq1$ and $L_{1}^{\nu}=L^{\nu}.$
\end{rem}

\section{Continuity of Operator $m\left(y\right)\nu\left(dy\right)$}

We now present $L_{p}-$continuity of operator $\pi$ when $\pi\prec\nu.$
In particular, we will apply the following Lemma to $\pi\left(dy\right)=m\left(y\right)\nu\left(dy\right)$
where $m$ is bounded. 
\begin{lem}
\noindent \label{lem:(Continuity-of-Operators)} Let $\nu\in\mathfrak{A}^{\sigma},w=w_{\nu}$
be a continuous O-RV function and \textbf{A, B} hold. Suppose that
$\pi\prec\nu\left(M,\alpha_{1},\alpha_{2}\right)$ then for each $p\in\left(1,\infty\right)$
there exists $N=N\left(d,p,\nu\right)>0$ such that 
\[
\left|L^{\pi}f\right|_{L_{p}}\leq NM\left|L^{\nu}f\right|_{L_{p}},\hspace{1em}f\in\tilde{C}^{\infty}\left(\mathbf{R}^{d}\right).
\]
\end{lem}
\begin{proof}
\noindent Let $\epsilon>0,$and $g=-\left(L^{\nu}-\epsilon\right)f$.
By \cite[Corollary 5]{MPh1}, 
\[
f\left(x\right)=\int_{0}^{\infty}e^{-\epsilon t}\mathbf{E}g\left(x+Z_{t}^{\nu}\right)dt,x\in\mathbf{R}^{d},
\]

\noindent and thus
\[
L^{\pi}f\left(x\right)=Tg\left(x\right):=\int_{0}^{\infty}e^{-\epsilon t}L^{\pi}\mathbf{E}g\left(x+Z_{t}^{\nu}\right)dt.
\]

\noindent Let 
\begin{align*}
T_{\epsilon}g\left(x\right) & =\int_{\epsilon}^{\infty}e^{-\epsilon t}L^{\pi}\mathbf{E}g\left(x+Z_{t}^{\nu}\right)dt\\
 & =\int m_{\epsilon}\left(x-y\right)g\left(y\right)dy
\end{align*}

\noindent where $m_{\epsilon}\left(x\right)=\int_{\epsilon}^{\infty}e^{-\epsilon t}L^{\pi}p^{\nu^{\ast}}\left(t,x\right)dt$. 

We now prove that there is $N>0$ independent of $\epsilon$ so that
\[
\left|T_{\epsilon}g\right|_{L_{p}}\leq NM\left|g\right|_{L_{p}},\epsilon>0,g\in L_{p}\left(\mathbf{R}^{d}\right).
\]

\noindent First we consider the Fourier transform of $m_{\epsilon}$.
By Lemma \ref{lem:symbolEst},
\begin{align*}
\left|\hat{m_{\epsilon}}\left(\xi\right)\right| & \leq\int_{\epsilon}^{\infty}\left|\psi^{\pi}\left(\xi\right)\right|\left|\exp\left\{ \psi^{\nu}\left(\xi\right)t-\epsilon t\right\} \right|dt\\
 & \leq NM\int_{\epsilon}^{\infty}w\left(\left|\xi\right|^{-1}\right)^{-1}\exp\left\{ -cw\left(\left|\xi\right|^{-1}\right)^{-1}t-\epsilon t\right\} dt\\
 & \leq NM.
\end{align*}

\noindent Now according to \cite[Theorem 3 in Chapter 1]{Ste}, it
suffices to show that

\noindent 
\[
\int_{\left|x\right|>3\left|y\right|}\left|m_{\epsilon}\left(x-y\right)-m_{\epsilon}\left(x\right)\right|dx\leq NM,y\neq0.
\]

We split the integral, 
\begin{align*}
 & \int_{\left|x\right|>3\left|y\right|}\left|\int_{\epsilon}^{\infty}e^{-\epsilon t}L^{\pi}p^{\nu^{\ast}}\left(t,x-y\right)-e^{-\epsilon t}L^{\pi}p^{\nu^{\ast}}\left(t,x\right)dt\right|dx\\
\leq & \int_{\left|x\right|>3\left|y\right|}\left|\int_{0}^{w\left(\left|y\right|\right)}...dt\right|dx+\int_{\left|x\right|>3\left|y\right|}\left|\int_{w\left(\left|y\right|\right)}^{\infty}...dt\right|dx\\
= & \mathcal{I}_{1}+\mathcal{I}_{2}.
\end{align*}

\noindent By (\ref{eq:LpiP}), Lemma \ref{lem:timeDensity} (ii) (with
$m=1$), Chebyshev inequality, and Corollary \ref{cor:aymp_a},

\noindent 
\begin{align*}
\mathcal{I}_{1} & \leq2\int_{\left|x\right|\geq2\left|y\right|}\int_{0}^{w\left(\left|y\right|\right)}\left|L^{\pi}p^{\nu^{\ast}}\left(t,x\right)\right|dtdx\\
 & \leq NM\left|y\right|^{-\alpha_{2}}\int_{0}^{w\left(\left|y\right|\right)}t^{-1}a\left(t\right)^{\alpha_{2}}dt\\
 & \leq NM\left|y\right|^{-\alpha_{2}}a\left(w\left(\left|y\right|\right)\right)^{\alpha_{2}}\leq NM.
\end{align*}

\noindent By the fundamental theorem of calculus, (\ref{eq:LpiP}),
Lemma \ref{lem:timeDensity} (ii), Corollary \ref{cor:aymp_a}, and
Lemma \ref{lem:index_a} (iii),
\begin{align*}
\mathcal{I}_{2} & \leq\int_{\left|x\right|\geq3\left|y\right|}\int_{w\left(y\right)}^{\infty}\left|L^{\pi}p^{\nu^{\ast}}\left(t,x-y\right)-L^{\pi}p^{\nu^{\ast}}\left(t,x\right)\right|dtdx\\
 & \leq\int_{w\left(y\right)}^{\infty}\left|y\right|\int\int_{0}^{1}\left|L^{\pi}\nabla p^{\nu^{\ast}}\left(t,x-ry\right)\right|drdxdt\\
 & \leq NM\left|y\right|\int_{w\left(y\right)}^{\infty}a\left(t\right)^{-1}t^{-1}dt\leq NM\left|y\right|a\left(w\left(y\right)\right)^{-1}\\
 & \leq NM.
\end{align*}

The proof is finished.
\end{proof}
\noindent 
\begin{rem}
\noindent \label{rem:continuity} Under assumptions of Lemma \ref{lem:(Continuity-of-Operators)},
let $\left|m\left(y\right)\right|\leq K,y\in\mathbf{R}^{d}$ where

\noindent $\int_{r\leq\left|y\right|\leq R}ym\left(y\right)^{+}\nu\left(dy\right)=0$
and $\int_{r\leq\left|y\right|\leq R}ym\left(y\right)^{-}\nu\left(dy\right)=0$
for any $0<r<R$ if $\sigma=1$. We may apply Lemma \ref{lem:(Continuity-of-Operators)}
to $\pi\left(dy\right)=m\left(y\right)^{+}\nu\left(dy\right)$ and
$\pi\left(dy\right)=m\left(y\right)^{-}\nu\left(dy\right)$ respectively.
We conclude that there is $N=N\left(d,p,\nu\right)>0$ so that
\[
\left|L^{m,\nu}f\right|_{L_{p}}\leq NK\left|L^{\nu}f\right|_{L_{p}},\hspace{1em}f\in\tilde{C}^{\infty}\left(\mathbf{R}^{d}\right).
\]
\end{rem}
\noindent 
\begin{rem}
\noindent \label{rem:norm_equi}Another implication of Lemma \ref{lem:(Continuity-of-Operators)}
is that $\left|L^{\nu;1}f\right|_{L_{p}}\asymp\left|L^{\nu}f\right|_{L_{p}},f\in\tilde{C}^{\infty}\left(\mathbf{R}^{d}\right)$
and therefore the equivalent norm of $H_{p}^{\nu}$ is given by 
\[
\left|f\right|_{H_{p}^{\nu}}=\left|f\right|_{L_{p}}+\left|L^{\nu}f\right|_{L_{p}}.
\]
\end{rem}
We now provide an alternative proof of Lemma \ref{lem:(Continuity-of-Operators)}
for isotropic unimodal Lévy measure $\nu\left(dy\right)=j_{d}\left(\left|y\right|\right)dy$
with $\sigma\in\left(0,1\right)$ for which \textbf{D} holds and $\pi=m\left(y\right)\nu\left(dy\right)$.
The reason that we include this option even though it is clearly more
restrictive is because a stronger $L_{\infty}$-type Hörmander condition
(\ref{eq:Cal-Zyg}) holds in this case. One benefit of this stronger
estimate is that we can immediately claim a weighted version of continuity
which is an indispensable ingredient if one want to study (\ref{eq:mainEq})
in weighted Bessel potential spaces by a continuation of parameter
approach (cf. \cite{DK}.) At the moment, we present the next result
as an independent interest. We defer a full study on weighted $L_{p}-$boundedness
with spatially independent variable to a forthcoming paper because
it requires further understanding of embedding between classical weighted
potential spaces and generalized ones.

We now recall the definition of Muckenhoupt weights. For $p\in\left(1,\infty\right)$,
we say that a non-negative function $\omega\in A_{p}\left(\mathbf{R}^{d}\right)$
if there is $C>0$ such that for all balls $B$ in $\mathbf{R}^{d}$
we have 
\[
\left(\frac{1}{\left|B\right|}\int_{B}\omega\left(x\right)dx\right)\left(\frac{1}{\left|B\right|}\int_{B}\omega\left(x\right)^{-1/\left(p-1\right)}dx\right)^{p-1}<C.
\]

In a standard way, we denote by $L_{p,\omega}=L_{p,\omega}\left(\mathbf{R}^{d}\right)$
the space of locally integrable functions for which $\left(\int\left|f\left(x\right)\right|^{p}\omega\left(x\right)dx\right)^{1/p}<\infty.$
\begin{lem}
\noindent Let $\nu\left(dy\right)=j_{d}\left(\left|y\right|\right)dy\in\mathfrak{A}^{\sigma}$
for some $\sigma\in\left(0,1\right)$ and \textbf{D} hold. Suppose
that $\left|m\left(y\right)\right|\leq K,y\in\mathbf{R}^{d}$ then
for each $p\in\left(1,\infty\right)$ there exists $N=N\left(d,p,\nu,K\right)>0$
such that 
\[
\left|L^{m,\nu}f\right|_{L_{p,\omega}}\leq N\left|L^{\nu}f\right|_{L_{p,\omega}},\hspace{1em}f\in\tilde{C}^{\infty}\left(\mathbf{R}^{d}\right).
\]
\end{lem}
\begin{proof}
\noindent Without loss of generality we assume $K=1$ (applying the
result to $m/K$.) Note also that we only deal with the case $\sigma\in\left(0,1\right).$ 

We set for $\epsilon\in\left(0,1\right)$, $m_{\epsilon}\left(y\right)=m\left(y\right)\chi_{\epsilon\leq\left|y\right|\leq\epsilon^{-1}}$
and for $f\in\tilde{C}^{\infty}\left(\mathbf{R}^{d}\right),x\in\mathbf{R}^{d},$ 

\noindent 
\begin{align*}
L^{m,\nu}f\left(x\right) & =\int\left[f\left(x+y\right)-f\left(x\right)\right]m\left(y\right)\nu\left(dy\right)\\
 & =\lim_{\epsilon\rightarrow0}\int\left[f\left(x+y\right)-f\left(x\right)\right]m_{\epsilon}\left(y\right)\nu\left(dy\right).
\end{align*}

\noindent Applying Lemma \ref{lem:diff_frac}, 
\begin{align*}
L^{m,\nu}f\left(x\right) & =c\lim_{\epsilon\rightarrow0}\int\int k^{\nu;1}\left(z,y\right)m_{\epsilon}\left(y\right)\nu\left(dy\right)L^{\nu}f\left(x-z\right)dz\\
 & =c\lim_{\epsilon\rightarrow0}\int k_{\epsilon}\left(z\right)L^{\nu}f\left(x-z\right)dz=c\lim_{\epsilon\rightarrow0}T_{\epsilon}L^{\nu}f.
\end{align*}

\noindent where 

\noindent 
\begin{align*}
k^{\nu;1}\left(z,y\right) & =w_{\nu}\left(\left|y\right|\right)\left|y\right|^{-d}\int_{0}^{\infty}\left[p^{\tilde{\nu}_{\left|y\right|}}\left(t,\frac{z}{\left|y\right|}+\hat{y}\right)-p^{\tilde{\nu}_{\left|y\right|}}\left(t,\frac{z}{\left|y\right|}\right)\right]dt\\
 & =w_{\nu}\left(\left|y\right|\right)\int_{0}^{\infty}\left[p^{\nu}\left(w_{\nu}\left(\left|y\right|\right)t,z+y\right)-p^{\nu}\left(w_{\nu}\left(\left|y\right|\right)t,z\right)\right]dt\\
 & =\int_{0}^{\infty}\left[p^{\nu}\left(t,z+y\right)-p^{\nu}\left(t,z\right)\right]dt
\end{align*}

\noindent and
\[
k_{\epsilon}\left(z\right)=\int\int_{0}^{\infty}\left[p^{\nu}\left(t,z+y\right)-p^{\nu}\left(t,z\right)\right]dtm_{\epsilon}\left(y\right)\nu\left(dy\right).
\]

We will use a well-known result in harmonic analysis (see \cite{C,CF})
to obtain a weighted $L_{p}$ continuity for $T_{\epsilon}$ according
to which it suffices to prove $L_{2}-$continuity and the following
Lipschitz condition: there are $N,\eta>0$ such that for any $\left|x\right|>4\left|z\right|,x,z\in\mathbf{R}^{d}$, 

\noindent 
\begin{equation}
\left|k_{\epsilon}\left(x+z\right)-k_{\epsilon}\left(x\right)\right|\leq N\frac{\left|z\right|^{\eta}}{\left|x\right|^{d+\eta}}.\label{eq:Cal-Zyg}
\end{equation}

All estimates will be independent of $\epsilon$. 

\noindent Step I. By Plancherel's identity and Lemma \ref{lem:symbolEst},
\begin{align*}
\left|T_{\epsilon}f\right|_{L_{2}} & =\left|\mathcal{F}T_{\epsilon}f\right|_{L_{2}}\\
 & =\left|\int\left(e^{i2\pi y\cdot\xi}-1\right)\int_{0}^{\infty}\exp\left\{ t\psi^{\nu^{\ast}}\left(\xi\right)\right\} dtm_{\epsilon}\left(y\right)\nu\left(dy\right)\hat{f}\left(\xi\right)\right|_{L_{2}}\\
 & \leq N\left|\int w_{\nu}\left(\left|\xi\right|^{-1}\right)\left|e^{i2\pi y\cdot\xi}-1\right|\nu\left(dy\right)\hat{f}\left(\xi\right)\right|_{L_{2}}\\
 & =N\left|\int\left|e^{i2\pi y\cdot\hat{\xi}}-1\right|\tilde{\nu}^{\left|\xi\right|^{-1}}\left(dy\right)\hat{f}\left(\xi\right)\right|_{L_{2}}\\
 & \leq N\left(\int_{\left|y\right|\leq1}\left|y\right|\tilde{\nu}^{\left|\xi\right|^{-1}}\left(dy\right)+\int_{\left|y\right|>1}\tilde{\nu}^{\left|\xi\right|^{-1}}\left(dy\right)\right)\left|f\right|_{L_{2}}\\
 & \leq N\left|f\right|_{L_{2}}.
\end{align*}

\noindent where we apply Lemma \ref{lem:alpha12_integral} in the
last inequality. 

\noindent Step II. We prove (\ref{eq:Cal-Zyg}). For this purpose,
we use the splitting from \cite[Lemma 4.5]{DK} which is a simplified
version of \cite[Lemma 15]{MPr14}. Both results provide proof of
this Lemma for $\alpha-$stable like measures. Corollary \ref{cor:diff_D}
allows us to generalize their results. 

We denote for convenience $r=\frac{\left|x\right|}{\left|z\right|}>4.$
\begin{align*}
 & \left|k_{\epsilon}\left(x+z\right)-k_{\epsilon}\left(x\right)\right|\\
 & =\left|\int\int_{0}^{\infty}\left[p^{\nu}\left(t,x+z+\left|x\right|y\right)-p^{\nu}\left(t,x+z\right)\right]dtm_{\epsilon}\left(\left|x\right|\left|y\right|\right)\left|x\right|^{d}j_{d}\left(\left|x\right|\left|y\right|\right)dy\right.\\
 & -\left.\int\int_{0}^{\infty}\left[p^{\nu}\left(t,x+\left|x\right|y\right)-p^{\nu}\left(t,x\right)\right]dtm_{\epsilon}\left(\left|x\right|\left|y\right|\right)\left|x\right|^{d}j_{d}\left(\left|x\right|\left|y\right|\right)dy\right|\\
 & \preceq\int_{\left|y\right|\leq\frac{1}{r}}\int_{0}^{\infty}\left|...\right|dtdy+\int_{\frac{1}{r}\leq\left|y\right|,\frac{1}{2}\leq\left|y+\hat{x}\right|}\int_{0}^{\infty}\left|...\right|dtdy\\
 & +\int_{\left|y+\hat{x}\right|\leq\frac{2}{r}}\int_{0}^{\infty}\left|...\right|dtdy+\int_{\frac{2}{r}\leq\left|y+\hat{x}\right|\leq\frac{1}{2}}\int_{0}^{\infty}\left|...\right|dtdy\\
 & :=\mathcal{I}_{1}+\mathcal{I}_{2}+\mathcal{I}_{3}+\mathcal{I}_{4}
\end{align*}

\noindent where
\[
\left|...\right|=\left|p^{\nu}\left(t,x+z+\left|x\right|y\right)-p^{\nu}\left(t,x+z\right)-p^{\nu}\left(t,x+\left|x\right|y\right)+p^{\nu}\left(t,x\right)\right|\left|x\right|^{d}j_{d}\left(\left|x\right|\left|y\right|\right).
\]

Before proceeding to the main proof, we note that by assumption \textbf{D},
\[
\left|x\right|^{d}j_{d}\left(\left|x\right|\left|y\right|\right)\asymp\frac{1}{\gamma\left(\left|x\right|\left|y\right|\right)\left|y\right|^{d}}.
\]

\noindent Also, we fix $0<\alpha_{2}\leq\alpha_{1}<1$ chosen as in
Lemma \ref{lem:alpha12_ratio} for $\gamma.$ 

\medskip

\textbf{Estimate of $\mathcal{I}_{1}$}: if $\left|y\right|\leq\frac{1}{r}$
then $\left|\hat{x}+ly\right|\geq\frac{1}{2}$ and $\left|\hat{x}+z/\left|x\right|+ly\right|\geq\frac{1}{2}$
for $l\in\left[0,1\right].$ 
\begin{align*}
\mathcal{I}_{1} & \preceq\int_{\left|y\right|\leq\frac{1}{r}}\int_{0}^{\infty}\left|p^{\nu}\left(t,x+\left|x\right|y\right)-p^{\nu}\left(t,x\right)\right|dt\frac{dy}{\gamma\left(\left|x\right|\left|y\right|\right)\left|y\right|^{d}}\\
 & +\int_{\left|y\right|\leq\frac{1}{r}}\int_{0}^{\infty}\left|p^{\nu}\left(t,x+z+\left|x\right|y\right)-p^{\nu}\left(t,x+z\right)\right|dt\frac{dy}{\gamma\left(\left|x\right|\left|y\right|\right)\left|y\right|^{d}}\\
 & =\mathcal{I}_{1,1}+\mathcal{I}_{1,2}.
\end{align*}

By Corollary \ref{cor:diff_D}, 

\noindent 
\begin{align*}
\mathcal{I}_{1,1} & \preceq\int_{0}^{1}\int_{\left|y\right|\leq\frac{1}{r}}\left|x\right|\frac{\gamma\left(\left|x+l\left|x\right|y\right|\right)}{\left|x+l\left|x\right|y\right|^{d+1}}\frac{dy}{\gamma\left(\left|x\right|\left|y\right|\right)\left|y\right|^{d-1}}dl\\
 & =\int_{0}^{1}\int_{\left|y\right|\leq\frac{1}{r}}\frac{\gamma\left(\left|x\right|\left|\hat{x}+ly\right|\right)}{\left|x\right|^{d}\left|\hat{x}+ly\right|^{d+1}}\frac{dy}{\gamma\left(\left|x\right|\left|y\right|\right)\left|y\right|^{d-1}}dl.
\end{align*}

We now apply Lemma \ref{lem:alpha12_ratio} to obtain 
\begin{align*}
\mathcal{I}_{1,1} & \preceq\int_{0}^{1}\int_{\left|y\right|\leq\frac{1}{r}}\frac{1}{\left|x\right|^{d}\left|\hat{x}+ly\right|^{d+1}}\left(\frac{\hat{x}+ly}{\left|y\right|}\right)^{\alpha_{1}}\frac{dy}{\left|y\right|^{d-1}}dl\\
 & =\int_{0}^{1}\int_{\left|y\right|\leq\frac{1}{r}}\frac{1}{\left|x\right|^{d}\left|\hat{x}+ly\right|^{d+1-\alpha_{1}}}\frac{1}{\left|y\right|^{d+\alpha_{1}-1}}dydl\preceq\frac{1}{\left|x\right|^{d}}\left(\frac{1}{r}\right)^{1-\alpha_{1}}\\
 & =\frac{\left|z\right|^{1-\alpha_{1}}}{\left|x\right|^{d+1-\alpha_{1}}}.
\end{align*}

The estimate for $\mathcal{I}_{1,2}$ is derived by essentially by
the same argument replacing $\hat{x}+ly$ with $\hat{x}+z/\left|x\right|+ly$.

\medskip

\textbf{Estimate of $\mathcal{I}_{2}$}: if $\left|y+\hat{x}\right|\geq\frac{1}{2}$
then $\left|\hat{x}+lz/\left|x\right|\right|\geq\frac{1}{4}$ and
$\left|\hat{x}+y+lz/\left|x\right|\right|\geq\frac{1}{4}$ for $l\in\left[0,1\right].$
\begin{align*}
 & \mathcal{I}_{2}\\
 & \preceq\int_{\frac{1}{r}\leq\left|y\right|,\frac{1}{2}\leq\left|y+\hat{x}\right|}\int_{0}^{\infty}\left|p^{\nu}\left(t,x+z+\left|x\right|y\right)-p^{\nu}\left(t,x+\left|x\right|y\right)\right|dt\frac{dy}{\gamma\left(\left|x\right|\left|y\right|\right)\left|y\right|^{d}}\\
 & +\int_{\frac{1}{r}\leq\left|y\right|,\frac{1}{2}\leq\left|y+\hat{x}\right|}\int_{0}^{\infty}\left|p^{\nu}\left(t,x+z\right)-p^{\nu}\left(t,x\right)\right|dt\frac{dy}{\gamma\left(\left|x\right|\left|y\right|\right)\left|y\right|^{d}}\\
 & =\mathcal{I}_{2,1}+\mathcal{I}_{2,2}.
\end{align*}

By Corollary \ref{cor:diff_D} and Lemma \ref{lem:alpha12_ratio}, 

\noindent 
\begin{align*}
 & \mathcal{I}_{2,2}\\
 & \preceq\left|z\right|\int_{0}^{1}\int_{\frac{1}{r}\leq\left|y\right|,\frac{1}{2}\leq\left|y+\hat{x}\right|}\frac{\gamma\left(\left|x+lz\right|\right)}{\left|x+lz\right|^{d+1}}\frac{dy}{\gamma\left(\left|x\right|\left|y\right|\right)\left|y\right|^{d}}dl\\
 & \preceq\frac{\left|z\right|}{\left|x\right|^{d+1}}\int_{0}^{1}\int_{\frac{1}{r}\leq\left|y\right|,\frac{1}{2}\leq\left|y+\hat{x}\right|}\frac{1}{\left|\hat{x}+lz/\left|x\right|\right|^{d+1}}\left[\left(\frac{\hat{x}+lz/\left|x\right|}{\left|y\right|}\right)^{\alpha_{2}}+\left(\frac{\hat{x}+lz/\left|x\right|}{\left|y\right|}\right)^{\alpha_{1}}\right]\frac{dy}{\left|y\right|^{d}}dl\\
 & \preceq\frac{\left|z\right|^{1-\alpha_{1}}}{\left|x\right|^{d+1-\alpha_{1}}}.
\end{align*}

The estimate for $\mathcal{I}_{2,1}$ is derived by essentially by
the same argument replacing $\hat{x}+lz/\left|x\right|$ with $\hat{x}+y+lz/\left|x\right|$.

\medskip

\textbf{Estimate of $\mathcal{I}_{3}:$} if $\left|y+\hat{x}\right|\leq\frac{2}{r}$
then $\left|y\right|\geq\frac{1}{2}$ and $\left|\hat{x}+z/\left|x\right|+y\right|\leq\frac{3}{r}.$

\noindent 
\begin{align*}
\mathcal{I}_{3} & \preceq\int_{\left|y+\hat{x}\right|\leq\frac{2}{r}}\int_{0}^{\infty}\left|p^{\nu}\left(t,x+z+\left|x\right|y\right)\right|dt\frac{dy}{\gamma\left(\left|x\right|\left|y\right|\right)\left|y\right|^{d}}\\
 & +\int_{\left|y+\hat{x}\right|\leq\frac{2}{r}}\int_{0}^{\infty}\left|p^{\nu}\left(t,x+z\right)\right|dt\frac{dy}{\gamma\left(\left|x\right|\left|y\right|\right)\left|y\right|^{d}}\\
 & +\int_{\left|y+\hat{x}\right|\leq\frac{2}{r}}\int_{0}^{\infty}\left|p^{\nu}\left(t,x+\left|x\right|y\right)\right|dt\frac{dy}{\gamma\left(\left|x\right|\left|y\right|\right)\left|y\right|^{d}}\\
 & +\int_{\left|y+\hat{x}\right|\leq\frac{2}{r}}\int_{0}^{\infty}\left|p^{\nu}\left(t,x\right)\right|dt\frac{dy}{\gamma\left(\left|x\right|\left|y\right|\right)\left|y\right|^{d}}\\
 & =\mathcal{I}_{3,1}+\mathcal{I}_{3,2}+\mathcal{I}_{3,3}+\mathcal{I}_{3,4}.
\end{align*}

By Corollary \ref{cor:diff_D}, Lemma \ref{lem:alpha12_ratio}, and
changing the variable of integration,
\begin{align*}
\mathcal{I}_{3,3} & \preceq\int_{\left|y+\hat{x}\right|\leq\frac{2}{r},\left|y\right|\geq\frac{1}{2}}\frac{\gamma\left(x+\left|x\right|y\right)}{\left|x+\left|x\right|y\right|^{d}}\frac{dy}{\gamma\left(\left|x\right|\left|y\right|\right)\left|y\right|^{d}}\\
 & \preceq\frac{1}{\left|x\right|^{d}}\int_{\left|y+\hat{x}\right|\leq\frac{2}{r},\left|y\right|\geq\frac{1}{2}}\frac{1}{\left|\hat{x}+y\right|^{d}}\left[\left(\frac{\left|\hat{x}+y\right|}{\left|y\right|}\right)^{\alpha_{2}}+\left(\frac{\left|\hat{x}+y\right|}{\left|y\right|}\right)^{\alpha_{1}}\right]\frac{dy}{\left|y\right|^{d}}\\
 & \preceq\frac{1}{\left|x\right|^{d}}\int_{\left|y\right|\leq\frac{2}{r}}\left(\frac{1}{\left|y\right|^{d-\alpha_{2}}}+\frac{1}{\left|y\right|^{d-\alpha_{1}}}\right)dy\preceq\frac{\left|z\right|^{\alpha_{2}}}{\left|x\right|^{d+\alpha_{2}}}.
\end{align*}

The estimate for $\mathcal{I}_{3,1}$ is derived by the same argument
for $\mathcal{I}_{3,3}$ where we replace $\left|\hat{x}+y\right|$
with $\left|\hat{x}+z/\left|x\right|+y\right|$ and $\frac{2}{r}$
with $\frac{3}{r}$ respectively. 

By Corollary \ref{cor:diff_D} and Lemma \ref{lem:alpha12_ratio},
\begin{align*}
\mathcal{I}_{3,4} & \preceq\frac{\gamma\left(\left|x\right|\right)}{\left|x\right|^{d}}\int_{\left|y+\hat{x}\right|\leq\frac{2}{r},\left|y\right|\geq\frac{1}{2}}\frac{dy}{\gamma\left(\left|x\right|\left|y\right|\right)\left|y\right|^{d}}\\
 & \preceq\frac{1}{\left|x\right|^{d}}\int_{\left|y+\hat{x}\right|\leq\frac{2}{r},\left|y\right|\geq\frac{1}{2}}\left(\frac{1}{\left|y\right|^{d+\alpha_{2}}}+\frac{1}{\left|y\right|^{d+\alpha_{1}}}\right)dy\\
 & \preceq\frac{1}{\left|x\right|^{d}}\int_{\left|y+\hat{x}\right|\leq\frac{2}{r}}dy\preceq\frac{\left|z\right|^{d}}{\left|x\right|^{2d}}.
\end{align*}

The estimate for $\mathcal{I}_{3,2}$ is derived by the same argument
for $\mathcal{I}_{3,4}$ where we replace $\left|\hat{x}+y\right|$
with $\left|\hat{x}+z/\left|x\right|+y\right|$ and $\frac{2}{r}$
with $\frac{3}{r}$ respectively. 

\medskip

\textbf{Estimate of $\mathcal{I}_{4}$}: if $\frac{2}{r}\leq\left|y+\hat{x}\right|\leq\frac{1}{2}$
then $\left|y\right|\geq\frac{1}{2}$ and for all $l\in\left[0,1\right]$
\[
\frac{1}{r}\leq\left|\hat{x}+lz/\left|x\right|+y\right|\leq1,\left|\hat{x}+lz/\left|x\right|\right|\geq\frac{1}{2}.
\]

\noindent 
\begin{align*}
\mathcal{I}_{4} & \preceq\int_{\frac{2}{r}\leq\left|y+\hat{x}\right|\leq\frac{1}{2}}\int_{0}^{\infty}\left|p^{\nu}\left(t,x+z\right)-p^{\nu}\left(t,x\right)\right|dt\frac{dy}{\gamma\left(\left|x\right|\left|y\right|\right)\left|y\right|^{d}}\\
 & +\int_{\frac{2}{r}\leq\left|y+\hat{x}\right|\leq\frac{1}{2}}\int_{0}^{\infty}\left|p^{\nu}\left(t,x+z+\left|x\right|y\right)-p^{\nu}\left(t,x+\left|x\right|y\right)\right|dt\frac{dy}{\gamma\left(\left|x\right|\left|y\right|\right)\left|y\right|^{d}}\\
 & =\mathcal{I}_{4,1}+\mathcal{I}_{4,2}.
\end{align*}

By Corollary \ref{cor:diff_D}, Lemma \ref{lem:alpha12_ratio}, and
changing the variable of integration, 

\noindent 
\begin{align*}
\mathcal{I}_{4,2} & \preceq\int_{0}^{1}\int_{\frac{1}{r}\leq\left|\hat{x}+lz/\left|x\right|+y\right|\leq1,\left|y\right|\geq\frac{1}{2}}\left|z\right|\frac{\gamma\left(\left|x+\left|x\right|y+lz\right|\right)}{\left|x+\left|x\right|y+lz\right|^{d+1}}\frac{dy}{\gamma\left(\left|x\right|\left|y\right|\right)\left|y\right|^{d}}dl\\
 & \preceq\int_{0}^{1}\int_{\frac{1}{r}\leq\left|\hat{x}+lz/\left|x\right|+y\right|\leq1,\left|y\right|\geq\frac{1}{2}}\frac{\left|z\right|}{\left|x\right|^{d+1}}\frac{1}{\left|\hat{x}+y+lz/\left|x\right|\right|^{d+1}}\\
 & \times\left[\left(\frac{\left|\hat{x}+y+lz/\left|x\right|\right|}{\left|y\right|}\right)^{\alpha_{1}}+\left(\frac{\left|\hat{x}+y+lz/\left|x\right|\right|}{\left|y\right|}\right)^{\alpha_{2}}\right]\frac{dy}{\left|y\right|^{d}}dl\\
 & \preceq\int_{0}^{1}\int_{\frac{1}{r}\leq\left|y\right|\leq1}\frac{\left|z\right|}{\left|x\right|^{d+1}}\left[\frac{1}{\left|y\right|^{d+1-\alpha_{1}}}+\frac{1}{\left|y\right|^{d+1-\alpha_{2}}}\right]dydl\\
 & \preceq\frac{\left|z\right|^{\alpha_{2}}}{\left|x\right|^{d+\alpha_{2}}}.
\end{align*}

The estimate for $\mathcal{I}_{4,1}$ is derived by essentially the
same argument for $\mathcal{I}_{4,2}$. The proof is complete by specifying
$\eta=\min\left\{ \alpha_{2},1-\alpha_{1}\right\} $ in (\ref{eq:Cal-Zyg}). 
\end{proof}

\section{Auxiliary Results}

\noindent Next we prove some essential embedding inequalities. 
\begin{lem}
\noindent \label{lem:(Interpolation-Inequalities)} Let $\nu\in\mathfrak{A}^{\sigma},w=w_{\nu}$
is a continuous O-RV function and \textbf{A, B} hold. 

\noindent (i) For any $\epsilon>0,\delta\in\left(0,1\right)$, there
exists $C=C\left(\epsilon,\delta,\nu\right)$ such that

\noindent 
\[
\left|L^{\nu;\delta}f\right|_{L_{p}}\leq\epsilon\left|L^{\nu}f\right|_{L_{p}}+C\left|f\right|_{L_{p}},f\in\mathcal{S}\left(\mathbf{R}^{d}\right).
\]

\noindent (ii) If $\sigma>1$, then for any $\epsilon>0$, there exists
$C=C\left(\epsilon,d,p,\nu\right)$ such that
\[
\left|\nabla f\right|_{L_{p}}\leq\epsilon\left|f\right|_{H_{p}^{\nu}}+C\left|f\right|_{L_{p}},f\in\mathcal{S}\left(\mathbf{R}^{d}\right).
\]
\end{lem}
\begin{proof}
\noindent (i) For $\delta\in\left(0,1\right)$ and $f\in\mathcal{S}\left(\mathbf{R}^{d}\right)$
we have 
\begin{align*}
-\left(-\psi^{\nu_{sym}}\left(\xi\right)\right)^{\delta}\hat{f}\left(\xi\right)= & c\int_{0}^{\infty}t^{-\delta}\left[\exp\left(\psi^{\nu_{sym}}\left(\xi\right)t\right)-1\right]\frac{dt}{t}\hat{f}\left(\xi\right),\xi\in\mathbf{R}^{d}
\end{align*}

\noindent and hence,
\begin{align*}
L^{\nu;\delta}f\left(x\right) & =\mathcal{F}^{-1}\left[-\left(-\psi^{\nu_{sym}}\right)^{\delta}\hat{f}\right]\left(x\right)\\
 & =c\mathbf{E}\int_{0}^{\infty}t^{-\delta}\left[f\left(x+Z_{t}^{\nu_{sym}}\right)-f\left(x\right)\right]\frac{dt}{t},x\in\mathbf{R}^{d}.
\end{align*}

By It\^{o} formula, for any $a>0$, 
\begin{align*}
L^{\nu;\delta}f\left(x\right) & =c\mathbf{E}\int_{0}^{a}t^{-\delta}\int_{0}^{t}L^{\nu_{sym}}f\left(x+Z_{r}^{\nu_{sym}}\right)dr\frac{dt}{t}\\
 & +c\mathbf{E}\int_{a}^{\infty}t^{-\delta}\left[f\left(x+Z_{t}^{\nu_{sym}}\right)-f\left(x\right)\right]\frac{dt}{t}.
\end{align*}

\noindent Therefore, for any $\epsilon_{1}>0$, we may choose $a$
sufficiently small such that for some $C\left(\epsilon_{1},\delta\right)>0$,
\[
\left|L^{\nu;\delta}f\right|_{L_{p}}\leq\epsilon_{1}\left|L^{\nu_{sym}}f\right|_{L_{p}}+C\left(\epsilon_{1},\delta\right)\left|f\right|_{L_{p}},f\in\mathcal{S}\left(\mathbf{R}^{d}\right).
\]

The statement follows from Lemma \ref{lem:(Continuity-of-Operators)}
with $\nu^{sym}\prec\nu$, choosing $\epsilon_{1}$ smaller if necessary. 

\noindent (ii) By an interpolation inequality for standard fractional
Laplacian for $1<\alpha<p_{1}\wedge p_{2}$ and each $\kappa\in\left(0,1\right)$
there is $C=C\left(\alpha,p,d\right)$ such that
\begin{equation}
\left|\nabla f\right|_{L_{p}}\leq\kappa\left|f\right|_{H_{p}^{\alpha}}+C\kappa^{-\frac{1}{\alpha-1}}\left|f\right|_{L_{p}}.\label{eq:int0}
\end{equation}

We estimate the first term further using a convenient equivalent norm
provided in \cite[Proposition 1]{MPh3} of Bessel potential spaces
with generalized smoothness. Let $\varphi_{j},j\geq0$ be a system
of functions as in defined in \cite[Remark 1]{MPh3} based on a partition
of unity and let $\gamma$ be such that $\alpha<\gamma<p_{1}\wedge p_{2}$.
Owing to Lemma \ref{lem:alpha12_ratio} we have, 
\begin{align}
\left|f\right|_{H_{p}^{\alpha}} & \leq N\left|\left(\sum_{j=0}^{\infty}\left|2^{\alpha j}\varphi_{j}\ast f\right|^{2}\right)^{1/2}\right|_{L_{p}}\label{eq:int1}\\
\leq & N\left[\left|\left(\sum_{j=0}^{n}\left|2^{\alpha j}\varphi_{j}\ast f\right|^{2}\right)^{1/2}\right|_{L_{p}}+\left|\left(\sum_{j=n}^{\infty}\left|2^{\left(\alpha-\gamma\right)j}w\left(2^{-j}\right)^{-1}\varphi_{j}\ast f\right|^{2}\right)^{1/2}\right|\right]\nonumber \\
\leq & N\left[2^{\alpha n}\left|f\right|_{L_{p}}+2^{\left(\alpha-\gamma\right)n}\left|f\right|_{H_{p}^{\nu}}\right].\nonumber 
\end{align}

\noindent The claim follows by combining (\ref{eq:int0}), (\ref{eq:int1})
and choosing $n$ sufficiently large. 
\end{proof}
The next Lemma is a generalization of \cite[Lemma 3]{MPr14}.
\begin{lem}
\noindent \label{lem:product0} Let $\nu\in\mathfrak{A}^{\sigma},w=w_{\nu}$
is a continuous O-RV function and \textbf{A, B} hold. Let $g,u\in\mathcal{S}\left(\mathbf{R}^{d}\right)$
and

\noindent 
\begin{align*}
J\left(x,l\right) & =\int\left[g\left(x+y\right)-g\left(x\right)\right]\left[u\left(x+y-l\right)-u\left(x-l\right)\right]m\left(x,y\right)\nu\left(dy\right),x,l\in\mathbf{R}^{d}\\
 & =\int_{\left|y\right|>1}...\nu\left(dy\right)+\int_{\left|y\right|\leq1}...\nu\left(dy\right)\\
 & =J_{1}\left(x,l\right)+J_{2}\left(x,l\right)
\end{align*}
\end{lem}
\noindent where $\left|m\left(x,y\right)\right|\leq K,x,y\in\mathbf{R}^{d}$.
Then there exists $0<\beta<1$ and $N=N\left(\beta,d,p,\nu\right)>0$
such that
\[
\int\int\left|J\left(x,l\right)\right|^{p}dxdl\leq NK^{p}\left(\left|g\right|_{L_{p}}^{p}+\left|L^{\nu;\beta}g\right|_{L_{p}}^{p}\right)\left|u\right|_{H_{p}^{1}}^{p}.
\]

\noindent Moreover, 
\[
\int\int\left|J\left(x,l\right)\right|^{p}dxdl\leq NK^{p}\left|g\right|_{H_{p}^{\nu;\beta}}^{p}\left|u\right|_{H_{p}^{1}}^{p}.
\]

\begin{proof}
\noindent We prove the first estimate starting with $J_{1}$. By Lemma
\ref{lem:alpha12_integral}, $\int_{\left|y\right|>1}\nu\left(dy\right)$
is finite and by Hölder inequality, 
\begin{align*}
 & \int\int\left|J_{1}\left(x,l\right)\right|^{p}dxdl\\
 & \leq NK^{p}\int\int\int_{\left|y\right|>1}\left|g\left(x+y\right)-g\left(x\right)\right|^{p}\left|u\left(x+y-l\right)-u\left(x-l\right)\right|^{p}\nu\left(dy\right)dxdl\\
 & \leq NK^{p}\left|g\right|_{L_{p}}^{p}\left|u\right|_{L_{p}}^{p}.
\end{align*}

If $\sigma\in\left(0,1\right)$, by Lemma \ref{lem:alpha12_integral},
$\int_{\left|y\right|\leq1}\left|y\right|\nu\left(dy\right)$ is finite
and by Hölder inequality,
\begin{align*}
 & \int\int\left|J_{2}\left(x,l\right)\right|^{p}dxdl\\
\leq & NK^{p}\int\int\int_{0}^{1}\int_{\left|y\right|\leq1}\left|g\left(x+y\right)-g\left(x\right)\right|^{p}\left|\nabla u\left(x+sy-l\right)\right|^{p}\left|y\right|\nu\left(dy\right)dsdxdl\\
\leq & NK^{p}\left|g\right|_{L_{p}}^{p}\left|\nabla u\right|_{L_{p}}^{p}.
\end{align*}

If $\sigma\in\left[1,2\right)$, we choose $\beta$ such that $\frac{\sigma-1}{\left(p_{1}\wedge p_{2}\right)}<\beta<1$.
Then by Lemma \ref{lem:diff_frac}, Minkowski inequality and Lemma
\ref{lem:gen-by-part},

\begin{align*}
 & \int\int\left|J_{2}\left(x,l\right)\right|^{p}dxdl\\
 & \leq NK^{p}(\int_{\left|y\right|\leq1}\int_{0}^{1}\int\left(\int\int\left|L^{\nu;\beta}g\left(x-z\right)\right|^{p}\left|\nabla u\left(x+sy-l\right)\right|^{p}dxdl\right)^{1/p}\\
 & \times\left|\kappa^{\nu;\beta}\left(z,y\right)\right|\left|y\right|dzds\nu\left(dy\right)){}^{p}\\
 & \leq NK^{p}\left|L^{\nu;\beta}g\right|_{L_{p}}^{p}\left|\nabla u\right|_{L_{p}}^{p}\left(\int_{\left|y\right|\leq1}\left|y\right|w\left(\left|y\right|\right)^{\beta}\nu\left(dy\right)\right)^{p}\\
 & \leq NK^{p}\left|L^{\nu;\beta}g\right|_{L_{p}}^{p}\left|\nabla u\right|_{L_{p}}^{p}.
\end{align*}

The second estimate follows from the first estimate and Remark \ref{rem:gen_bessel}.
\end{proof}
\noindent 
For the reader's convenience, we present some results from \cite{MPr14}
whose proof clearly carry over to our settings. The proof is an application
of Sobolev embedding theorem and Remark \ref{rem:continuity}.
\begin{lem}
\noindent \label{lem:emb2} (\cite[Lemma 6]{MPr14}) Let $\nu\in\mathfrak{A}^{\sigma},w=w_{\nu}$
be a continuous O-RV function and \textbf{A, B} hold. Let $\beta\in\left(0,1\right),p>d/\beta,\phi,f\in S\left(\mathbf{R}^{d}\right).$ 

\noindent (i) If $\left|m\left(z,y\right)\right|+\sup_{z\in\mathbf{R}^{d}}\left|\partial_{z}^{\beta}m\left(z,y\right)\right|\leq K,y\in\mathbf{R}^{d},$
then there exists $N=N\left(\beta,d,p,\nu\right)>0$ independent of
$\phi,f$ such that 
\[
\left|\phi L^{m,\nu}f\right|_{L_{p}}^{p}\leq NK^{p}\left|f\right|_{H_{p}^{\nu}}^{p}\left[\left|\phi\right|_{L_{p}}^{p}+\left|\nabla\phi\right|_{L_{p}}^{p}\right].
\]

\noindent (ii) If $m\left(z,y\right)$ is bounded and for a continuous
increasing function $\kappa\left(r\right),r>0,$
\[
\left|m\left(z,y\right)-m\left(z^{\prime},y\right)\right|\leq\kappa\left(\left|z-z^{\prime}\right|\right),\hspace{1em}z,z^{\prime},y\in\mathbf{R}^{d},
\]

\noindent with
\[
\int_{\left|y\right|\leq1}\kappa\left(\left|y\right|\right)\frac{dy}{\left|y\right|^{d+\beta}}<\infty,
\]

\noindent then there exists a constant $N=N\left(\beta,d,p,\nu\right)>0$
such that for $\epsilon\in\left(0,1\right],$ 
\[
\left|\phi L^{m,\nu}f\right|_{L_{p}}^{p}\leq NK\left(\epsilon,\phi\right)\left|f\right|_{H_{p}^{\nu}}^{p}
\]

\noindent where 
\begin{align*}
K\left(\epsilon,\phi\right) & =\epsilon^{-\beta p}\left|\phi\sup_{y}\left|m\left(\cdot,y\right)\right|\right|_{L_{p}}^{p}+\epsilon^{\left(1-\beta\right)p}\left|\sup_{y}\left|m\left(\cdot,y\right)\right|\nabla\phi\right|_{L_{p}}^{p}\\
 & +\kappa\left(\epsilon\right)^{p}\epsilon^{\left(1-\beta\right)p}\left|\nabla\phi\right|_{L_{p}}^{p}+l\left(\epsilon\right)^{p}\left|\phi\right|_{L_{p}}^{p}+\epsilon^{p}l\left(\epsilon\right)^{p}\left|\nabla\phi\right|_{L_{p}}^{p},\\
l\left(\epsilon\right) & =\int_{\left|y\right|\leq\epsilon}\kappa\left(\left|y\right|\right)\frac{dy}{\left|y\right|^{d+\beta}}.
\end{align*}
\end{lem}
\noindent 

\section{Equations without the Space Variable }

In this section, we study a special case of (\ref{eq:mainEq}) when
$m=m\left(t,y\right)$ is independent of $x$,
\begin{eqnarray}
du\left(t,x\right) & = & \left[L^{m,\nu}u\left(t,x\right)-\lambda u\left(t,x\right)+f\left(t,x\right)\right]dt\label{mainEq-noX}\\
u\left(0,x\right) & = & 0,\left(t,x\right)\in E=\left[0,T\right]\times\mathbf{R}^{d},\nonumber 
\end{eqnarray}

\noindent where $L^{m,\nu}=\int\left[\varphi\left(x+y\right)-\varphi\left(x\right)-\chi_{\sigma}\left(y\right)y\cdot\nabla\varphi\left(x\right)\right]m\left(t,y\right)\nu\left(dy\right),\varphi\in C_{0}^{\infty}\left(\mathbf{R}^{d}\right).$ 
\begin{lem}
\noindent \label{lem:EU} Let $\nu\in\mathfrak{A}^{\sigma}$, $\left|m\left(t,y\right)\right|\leq K,t\in\left[0,T\right],y\in\mathbf{R}^{d},$
$Z_{t}^{m,\nu}$ be defined in (\ref{f10}). Let $f\in\tilde{C}_{p}^{\infty}\left(E\right),p>1$,
then 
\begin{align*}
T_{\lambda}f\left(t,x\right) & =\int_{0}^{t}e^{-\lambda\left(t-s\right)}\mathbf{E}f\left(s,x+Z_{t}^{m,\nu}-Z_{s}^{m,\nu}\right)ds
\end{align*}
\end{lem}
\noindent is the unique solution in $\tilde{C}_{p}^{\infty}\left(E\right)$
of (\ref{mainEq-noX}). Moreover, for any multiindex $\gamma\in\mathbf{N}_{0}^{d}$,
\begin{equation}
\left|T_{\lambda}D^{\gamma}f\right|_{L_{p}\left(E\right)}\leq\rho_{\lambda}\left|D^{\gamma}f\right|_{L_{p}\left(E\right)}\label{eq:simpleEST}
\end{equation}

\noindent where $\rho_{\lambda}=T\wedge\frac{1}{\lambda}.$
\begin{proof}
\noindent Uniqueness follows from the argument of \cite[Lemma 7]{MPr14}.
For existence, we adapt the argument of \cite[Lemma 8]{MPh1} with
obvious changes i.e., applying It\^{o} formula (for a semi-martingale
e.g. \cite[Proposition 8.19]{CT}) with $e^{-\lambda r}f\left(s,x+\cdot\right)$
and $Z_{r}^{m,\nu}-Z_{s}^{m,\nu},s\leq r\leq t$ and afterward taking
the expectation and integrating in $s$. The estimate (\ref{eq:simpleEST})
follows from Hölder inequality.
\end{proof}
The goal now is to derive the main estimate,
\[
\left|L^{\nu}T_{\lambda}f\right|_{L_{p}\left(E\right)}\leq N\left|f\right|_{L_{p}\left(E\right)}.
\]

\begin{prop}
\noindent \label{prop:ME} Let $\nu\in\mathfrak{A}^{\sigma},w=w_{\nu}$
be a continuous O-RV function and \textbf{A, B} hold. Let $0<k\leq m\left(t,y\right)\leq K,t\in\left[0,T\right],y\in\mathbf{R}^{d}.$
Then for any $p>1$, there exists $N=N\left(d,p,\nu,k,K\right)$ so
that
\[
\left|L^{\nu}T_{\lambda}f\right|_{L_{p}\left(E\right)}\leq N\left|f\right|_{L_{p}\left(E\right)},f\in\tilde{C}_{p}^{\infty}\left(E\right).
\]
\end{prop}
\noindent 
\begin{proof}
\noindent We will again follow a classical route of Calderon-Zygmund
theorem for $L_{p}-$ boundedness of singular value integrals. Clearly,
it suffices to show the boundedness of 

\[
T_{\lambda}^{\epsilon}f\left(t,x\right)=\chi_{\epsilon<t}\int_{0}^{t-\epsilon}e^{-\lambda\left(t-s\right)}\mathbf{E}f\left(s,x+Z_{t}-Z_{s}\right)ds
\]

\noindent with estimates independent of $\epsilon\in\left(0,1\right).$
In particular, we prove the estimates in two steps. 

\noindent Step I. We prove the estimate for $p=2$, indeed applying
Plancherel's theorem, and by Lemma \ref{lem:symbolEst}, there are
$N,c>0$ such that $\left|\psi^{\nu}\left(\xi\right)\right|\leq Nw\left(\left|\xi\right|^{-1}\right)^{-1}$
and $\left|\psi^{m,\nu}\left(r,\xi\right)\right|\geq cw\left(\left|\xi\right|^{-1}\right)^{-1},r\in\left[0,T\right],\xi\in\mathbf{R}^{d}.$
Hence, by Minkowski inequality, 
\begin{align*}
 & \left|L^{\nu}T_{\lambda}^{\epsilon}f\right|_{L_{2}\left(E\right)}^{2}\\
 & \leq\left|\psi^{\nu}\left(\xi\right)\int_{0}^{t-\epsilon}\exp\left\{ \int_{s}^{t}\psi^{m,\nu}\left(r,\xi\right)dr-\lambda\left(t-s\right)\right\} \hat{f}\left(s,\xi\right)ds\right|_{L_{2}\left(E\right)}^{2}\\
 & \leq N\int\int_{0}^{T}\left(w\left(\left|\xi\right|^{-1}\right)^{-1}\int_{0}^{t-\epsilon}\exp\left\{ c\left(t-s\right)w\left(\left|\xi\right|^{-1}\right)^{-1}\right\} \left|\hat{f}\left(s,\xi\right)\right|ds\right)^{2}dtd\xi\\
 & =N\int\int_{0}^{T}\left(w\left(\left|\xi\right|^{-1}\right)^{-1}\int_{0}^{t}\exp\left\{ csw\left(\left|\xi\right|^{-1}\right)^{-1}\right\} \left|\hat{f}\left(t-s,\xi\right)\right|ds\right)^{2}dtd\xi\\
 & \leq N\int w\left(\left|\xi\right|^{-1}\right)^{-2}\left(\int_{0}^{T}\exp\left\{ csw\left(\left|\xi\right|^{-1}\right)^{-1}\right\} \left(\int_{s}^{T}\left|\hat{f}\left(t-s,\xi\right)\right|^{2}dt\right)^{1/2}ds\right)^{2}d\xi\\
 & \leq N\left|f\right|_{L_{2}\left(E\right)}^{2}.
\end{align*}

\noindent Step II. We prove the estimate for $p>2$ by verifying a
suitable Hörmander condition. Note that the estimate for $1<p<2$
follows from a standard duality argument. 

We apply \cite[Theorem 3.8]{KKK} which states that it is sufficient
to verify a Hörmander-type condition (\ref{eq:hor0}). It is worth
mentioning that differentiability assumptions in \cite[Assumption 4.1]{KKK}
prevents direct application of their Hörmander condition even in the
case of $\lambda=0.$ That said, we use some of their ideas in \cite[Corollary 4.11]{KKK}
for splitting of integrals. 

First we specify $c\geq1$ such that 
\begin{equation}
a\left(x+y\right)\leq c\left(a\left(x\right)+a\left(y\right)\right),x,y>0.\label{eq:triA}
\end{equation}
which is possible due to Lemma \ref{lem:index_a} (ii). In particular,
we have 
\begin{equation}
a\left(2x\right)\leq2ca\left(x\right),x>0.\label{eq:doubleA}
\end{equation}

Define for $\left(t,x\right),\left(s,y\right)\in E,r>0$ a region
\[
\mathcal{R}=\mathcal{R}\left(t,x,s,y,r\right)=\left\{ z\in\mathbf{R}^{d}:a\left(\left|t-r\right|\right)+\left|x-z\right|\geq4c\left(a\left(\left|t-s\right|\right)+\left|x-y\right|\right)\right\} .
\]

We must show that there exists $N>0$ such that for all $\left(t,x\right)$,$\left(s,y\right)\in E,$ 

\noindent 
\begin{align}
 & \int_{0}^{\infty}\int_{\mathcal{R}}\left|\chi_{0<r<t-\epsilon}e^{-\lambda\left(t-r\right)}L^{\nu}p^{m^{\ast},\nu^{\ast}}\left(r,t,x-z\right)-\chi_{0<r<s-\epsilon}e^{-\lambda\left(s-r\right)}L^{\nu}p^{m^{\ast},\nu^{\ast}}\left(r,s,y-z\right)\right|dzdr\nonumber \\
\leq & N.\label{eq:hor0}
\end{align}

We now prove (\ref{eq:hor0}) assume without loss of generality that
$0<\epsilon<s\leq t$ and $2s-t>0.$ 

\noindent 
\begin{align*}
 & \int_{0}^{\infty}\int_{\mathcal{R}}\left|\chi_{0<r<t-\epsilon}e^{-\lambda\left(t-r\right)}L^{\nu}p^{m^{\ast},\nu^{\ast}}\left(r,t,x-z\right)-\chi_{0<r<s-\epsilon}e^{-\lambda\left(s-r\right)}L^{\nu}p^{m^{\ast},\nu^{\ast}}\left(r,s,y-z\right)\right|dzdr\\
= & \int_{2s-t}^{t}\int_{\mathcal{R}}...+\int_{0}^{2s-t}\int_{\mathcal{R}}...=\mathcal{I}{}_{1}+\mathcal{I}_{2}.
\end{align*}

\noindent \textbf{Estimate of $\mathcal{I}_{1}$}: if $2s-t<r<t$
then in $\mathcal{R}$, $\left(t-r\right)<2\left(t-s\right)$, and
by (\ref{eq:doubleA}) 

\noindent 
\begin{align*}
\left|x-z\right| & \geq4c\left(a\left(\left|t-s\right|\right)+\left|x-y\right|\right)-a\left(2\left|t-s\right|\right)\\
 & \geq2a\left(\left|t-s\right|\right),
\end{align*}
and 
\[
\left|y-z\right|\geq\left|x-z\right|-\left|x-y\right|\geq2a\left(\left|t-s\right|\right).
\]

\noindent Therefore, by Lemma \ref{lem:prelim_hormander} (i) and
(\ref{eq:doubleA}), 

\noindent 
\begin{align*}
\mathcal{I}_{1} & \leq\int_{2s-t}^{t}\int_{\left|x-z\right|\geq2a\left(\left|t-s\right|\right)}\left|L^{\nu}p^{m^{\ast},\nu^{\ast}}\left(r,t,x-z\right)\right|dzdr\\
 & +\int_{2s-t}^{t}\int_{\left|y-z\right|\geq2a\left(\left|t-s\right|\right)}\left|L^{\nu}p^{m^{\ast},\nu^{\ast}}\left(r,s,y-z\right)\right|dzdr\\
 & \leq N2^{-\beta}a\left(t-s\right)^{-\beta}a\left(2\left(t-s\right)\right)^{\beta}\\
 & \leq N.
\end{align*}

\noindent \textbf{Estimate of $\mathcal{I}_{2}$}: 

\noindent 
\begin{align*}
 & \mathcal{I}_{2}\\
 & =\int_{0}^{2s-t}\int_{\mathcal{R}}\left|\chi_{0<r<t-\epsilon}e^{-\lambda\left(t-r\right)}L^{\nu}p^{m^{\ast},\nu^{\ast}}\left(r,t,x-z\right)-\chi_{0<r<t-\epsilon}e^{-\lambda\left(t-r\right)}L^{\nu}p^{m^{\ast},\nu^{\ast}}\left(r,t,y-z\right)\right|dzdr\\
 & +\int_{0}^{2s-t}\int_{\mathcal{R}}\left|\chi_{0<r<t-\epsilon}e^{-\lambda\left(t-r\right)}L^{\nu}p^{m^{\ast},\nu^{\ast}}\left(r,t,y-z\right)-\chi_{0<r<s-\epsilon}e^{-\lambda\left(s-r\right)}L^{\nu}p^{m^{\ast},\nu^{\ast}}\left(r,s,y-z\right)\right|dzdr\\
 & =\mathcal{I}_{2,1}+\mathcal{I}_{2,2}.
\end{align*}

\medskip

We consider $\mathcal{I}_{2,1}$ in two cases.

\noindent case i: $\left|x-y\right|\leq a\left(2\left(t-s\right)\right)$,
by Lemma \ref{lem:prelim_hormander} (ii)

\noindent 
\[
\mathcal{I}_{2,1}\leq N\left|x-y\right|a\left(2\left(t-s\right)\right)^{-1}\leq N.
\]

\noindent case ii: $\left|x-y\right|>a\left(2\left(t-s\right)\right)$
then 
\begin{equation}
w\left(\left|x-y\right|\right)\geq w\left(a\left(2\left(t-s\right)\right)\right)=2\left(t-s\right).\label{eq:wts}
\end{equation}

In this case, we split $\mathcal{I}_{2,1}$ further. 
\begin{align*}
\mathcal{I}_{2,1} & \leq\int_{0}^{\left(s-w\left(\left|x-y\right|\right)\right)\vee0}\int_{\mathcal{R}}...+\int_{\left(s-w\left(\left|x-y\right|\right)\right)\vee0}^{t}\int_{\mathcal{R}}...\\
 & =\mathcal{I}_{2,1,1}+\mathcal{I}_{2,1,2}.
\end{align*}

\noindent By Lemma \ref{lem:prelim_hormander} (ii), and Lemma \ref{lem:index_a}
(i),
\begin{align*}
\mathcal{I}_{2,1,1} & \leq N\left|x-y\right|a\left(t-s+w\left(\left|x-y\right|\right)\right)^{-1}\\
 & \leq N\left|x-y\right|a\left(w\left(\left|x-y\right|\right)+\right)^{-1}\\
 & \leq N.
\end{align*}

We now estimate $\mathcal{I}_{2,1,2}$. First, we observe that if
$s-w\left(\left|x-y\right|\right)<r<t$ then by (\ref{eq:wts}), (\ref{eq:doubleA}),
and Lemma \ref{lem:index_a} (i), 
\begin{align*}
a\left(t-r\right) & \leq a\left(t-s+w\left(\left|x-y\right|\right)\right)\leq a\left(\frac{3}{2}w\left(\left|x-y\right|\right)\right)\\
 & \leq2ca\left(\frac{3}{4}w\left(\left|x-y\right|\right)\right)\leq2c\left|x-y\right|.
\end{align*}

\noindent Therefore, we have in $\mathcal{R},$ 
\begin{align*}
2c\left|x-y\right|+\left|x-z\right|\geq a\left(\left|t-r\right|\right)+\left|x-z\right| & \geq4c\left(a\left(\left|t-s\right|\right)+\left|x-y\right|\right)
\end{align*}

\noindent which implies 
\[
\left|x-z\right|\geq4ca\left(\left|t-s\right|\right)+2c\left|x-y\right|\geq a\left(\left|t-s\right|\right)+\left|x-y\right|
\]
 and also
\[
\left|y-z\right|\geq\left|x-z\right|-\left|x-y\right|\geq a\left(\left|t-s\right|\right)+\left|x-y\right|.
\]

Hence, by Lemma \ref{lem:prelim_hormander} (i), (\ref{eq:triA})
and Lemma \ref{lem:index_a} (i), 
\begin{align*}
\mathcal{I}_{2,1,2} & \leq2\int_{\left(s-w\left(\left|x-y\right|\right)\right)\vee0}^{t}\int_{\left|z\right|\geq a\left(\left|t-s\right|\right)+\left|x-y\right|}\left|L^{\nu}p\left(r,t,z\right)\right|dzdr\\
 & \leq N\left(a\left(\left|t-s\right|\right)+\left|x-y\right|\right)^{-\beta}a\left(t-s+w\left(\left|x-y\right|\right)\right)^{\beta}\\
 & \leq N\left(a\left(\left|t-s\right|\right)+\left|x-y\right|\right)^{-\beta}\left(a\left(\left|t-s\right|\right)+\left|x-y\right|\right)^{\beta}\\
 & \leq N.
\end{align*}

Finally, we proceed to estimate $\mathcal{I}_{2,2}.$
\begin{align*}
 & \mathcal{I}_{2,2}\\
 & \leq\int_{0}^{2s-t}\int_{\mathcal{R}}\left|\chi_{0<r<t-\epsilon}e^{-\lambda\left(t-r\right)}L^{\nu}p^{m^{\ast},\nu^{\ast}}\left(r,t,y-z\right)-\chi_{0<r<s-\epsilon}e^{-\lambda\left(s-r\right)}L^{\nu}p^{m^{\ast},\nu^{\ast}}\left(r,s,y-z\right)\right|dzdr\\
 & =\int_{0}^{\left(s-\epsilon\right)\wedge\left(2s-t\right)}\int_{\mathcal{R}}...+\int_{\left(s-\epsilon\right)\wedge\left(2s-t\right)}^{2s-t}\int_{\mathcal{R}}...\\
 & =\mathcal{I}_{2,2,1}+\mathcal{I}_{2,2,2}.
\end{align*}

\noindent By Lemma \ref{lem:prelim_hormander} (iii) , (\ref{eq:LpiP})
and Lemma \ref{lem:timeDensity} (ii),
\begin{align*}
\mathcal{I}_{2,2,1}\leq & \int_{0}^{2s-t}\int e^{-\lambda\left(t-r\right)}\left|L^{\nu}p^{m^{\ast},\nu^{\ast}}\left(r,t,y-z\right)-L^{\nu}p^{m^{\ast},\nu^{\ast}}\left(r,s,y-z\right)\right|dzdr\\
+ & \int_{0}^{2s-t}\int\left|e^{-\lambda\left(t-r\right)}-e^{-\lambda\left(s-r\right)}\right|\left|L^{\nu}p^{m^{\ast},\nu^{\ast}}\left(r,s,y-z\right)\right|dzdr\\
\leq & N\left(t-s\right)\left(t-s\right)^{-1}+N\left|e^{-\lambda\left(t-s\right)}-1\right|\int_{t-s}^{s}\frac{e^{-\lambda r}}{r}dr\\
\leq & N+N\left(1\wedge\lambda\left(t-s\right)\right)\int_{\lambda\left(t-s\right)}^{\infty}\frac{e^{-r}}{r}dr.
\end{align*}

\noindent By elementary computation, one can easily check that the
function $f\left(x\right)=\left(1\wedge x\right)\int_{x}^{\infty}\frac{e^{-r}}{r}dr,x>0$
is bounded and thus $\mathcal{I}_{2,2,1}\leq N.$ 

For $\mathcal{I}_{2,2,2}$, we may assume $s-\epsilon<2s-t$ and consider
$\mathcal{I}_{2,2,2}$ in two cases. 

\noindent case i. $t-\epsilon>2s-t$ then $2\left(t-s\right)>\epsilon$,
by (\ref{eq:LpiP}) and Lemma \ref{lem:timeDensity} (ii), 

\begin{align*}
\mathcal{I}_{2,2,2}\leq & \int_{s-\epsilon}^{2s-t}\int\left|e^{-\lambda\left(t-r\right)}L^{\nu}p^{m^{\ast},\nu^{\ast}}\left(r,t,y-z\right)\right|dzdr\\
\leq & \int_{s-\epsilon}^{2s-t}\frac{e^{-\lambda\left(t-r\right)}}{t-r}dr=\int_{2\left(t-s\right)}^{t-s+\epsilon}\frac{e^{-\lambda r}}{r}dr\leq\ln\left(\frac{t-s+\epsilon}{2\left(t-s\right)}\right)\leq\ln\left(3/2\right).
\end{align*}

\noindent case ii. $t-\epsilon\leq2s-t$ then $2\left(t-s\right)\leq\epsilon$,
by (\ref{eq:LpiP}) and Lemma \ref{lem:timeDensity} (ii), 
\begin{align*}
\mathcal{I}_{2,2,2}\leq & \int_{s-\epsilon}^{t-\epsilon}\int\left|e^{-\lambda\left(t-r\right)}L^{\nu}p^{m^{\ast},\nu^{\ast}}\left(r,t,y-z\right)\right|dzdr\\
\leq & \int_{s-\epsilon}^{t-\epsilon}\frac{e^{-\lambda\left(t-r\right)}}{t-r}dr=\int_{\epsilon}^{t-s+\epsilon}\frac{e^{-\lambda r}}{r}dr\leq\ln\left(\frac{t-s+\epsilon}{\epsilon}\right)\leq\ln\left(3/2\right).
\end{align*}

The proof is complete for $0<\epsilon<s\leq t$ and $2s-t>0.$ The
statement for other cases follows from basically the same but simpler
argument. 
\end{proof}
We conclude this section with the existence and uniqueness result
for a general input functions.
\begin{lem}
\noindent Let $\nu\in\mathfrak{A}^{\sigma},w=w_{\nu}$ be a continuous
O-RV function and \textbf{A, B} hold. Let $p>1,$ $0<k\leq m\left(t,y\right)\leq K,t\in\left[0,T\right],y\in\mathbf{R}^{d}$.
Then for each $f\in L_{p}\left(E\right)$ there exists a unique solution
$u\in\mathcal{H}_{p}^{\nu}\left(E\right)$ of (\ref{mainEq-noX}).
Moreover there exist $N=N\left(d,p,\nu,k,K\right)$ such that 
\[
\left|L^{\nu}u\right|_{L_{p}\left(E\right)}\leq N\left|f\right|_{L_{p}\left(E\right)}
\]
and 
\[
\left|u\right|_{L_{p}\left(E\right)}\leq\rho_{\lambda}\left|f\right|_{L_{p}\left(E\right)}.
\]
\end{lem}
\begin{proof}
\noindent There exists a sequence $f_{n}\in\tilde{C}_{p}^{\infty}\left(E\right)$
such that $f_{n}\rightarrow f$ in $L_{p}\left(E\right)$. Applying
Lemma \ref{lem:EU} and Proposition \ref{prop:ME}, denoting the solution
of (\ref{mainEq-noX}) with $f=f_{n}$ by $u_{n}$, we have
\[
\left|u_{n}-u_{m}\right|_{L_{p}\left(E\right)}\leq\rho_{\lambda}\left|f_{n}-f_{m}\right|_{L_{p}\left(E\right)}\rightarrow0,\hspace{1em}m>n\rightarrow\infty.
\]
\[
\left|L^{\nu}u_{n}-L^{\nu}u_{m}\right|_{L_{p}\left(E\right)}\leq N\left|f_{n}-f_{m}\right|_{L_{p}\left(E\right)}\rightarrow0,\hspace{1em}m>n\rightarrow\infty.
\]

\noindent Since $H_{p}^{\nu}\left(E\right)$ is a Banach space, $u_{n}\rightarrow u\in H_{p}^{\nu}\left(E\right)$.
Passing to the limit, $u$ is a solution of (\ref{mainEq-noX}) with
$f$. Moreover, solution estimates follow from Lemma \ref{lem:EU}
and Proposition \ref{prop:ME}. Uniqueness can be proved as in proof
of \cite[Theorem 2]{MPh1} or \cite[Proposition 1]{MPr14} with minor
changes. 
\end{proof}

\section{Proof of the Main Theorem}

The proof of the main theorem is based on Lemma 8, Lemma 9 and Corollary
3 of \cite{MPr14} with our generalized auxiliary results in previous
sections in place of the original versions. We therefore only provide
a sketch of the proof omitting repetitive details and preserve some
notations from the original proof as needed for reader's convenience. 

\medskip

\noindent \textbf{Proof of Theorem} \ref{Thm:main}.
\begin{proof}
\noindent Step I. (\cite[Lemma 8]{MPr14}) We derive apriori estimate
assuming that the solution is compactly supported. Suppose that $u\in\tilde{C}_{p}^{\infty}\left(E\right)$
satisfying (\ref{Thm:main}) with a support of a ball of radius $\epsilon$
centered at $x_{0}$ 
\begin{align*}
\partial_{t}u & =L_{t,x_{0}}^{m,\nu}u\left(t,x\right)+L_{t,x}^{m,\nu}u\left(t,x\right)-L_{t,x_{0}}^{m,\nu}u\left(t,x\right)-\lambda u+f\\
u\left(0\right) & =0.
\end{align*}

Following the computation in the original proof and applying Lemma
\ref{lem:EU}, Proposition \ref{prop:ME}, Lemma \ref{lem:emb2} (ii),
and Remark \ref{rem:gen_bessel}, we derive

\noindent 
\begin{align*}
 & \left|L^{\nu}u\right|_{L_{p}\left(E\right)}\\
 & \leq N\left[\left|f\right|_{L_{p}\left(E\right)}+K\left(\epsilon\right)\left|u\right|_{H_{p}^{\nu}\left(E\right)}+w\left(\epsilon\right)^{-1}\left(\left|u\right|_{L_{p}\left(E\right)}+1_{\sigma>1}\left|\nabla u\right|_{L_{p}\left(E\right)}\right)\right]\\
 & \leq N\left[\left|f\right|_{L_{p}\left(E\right)}+K\left(\epsilon\right)C\left|L^{\nu}u\right|_{L_{p}\left(E\right)}+\left(w\left(\epsilon\right)^{-1}+C\right)\left(\left|u\right|_{L_{p}\left(E\right)}+1_{\sigma>1}\left|\nabla u\right|_{L_{p}\left(E\right)}\right)\right]
\end{align*}
and

\noindent 
\begin{align*}
 & \left|u\right|_{L_{p}\left(E\right)}\\
 & \leq\rho_{\lambda}\left[\left|f\right|_{L_{p}\left(E\right)}+K\left(\epsilon\right)C\left|L^{\nu}u\right|_{L_{p}\left(E\right)}+\left(w\left(\epsilon\right)^{-1}+C\right)\left(\left|u\right|_{L_{p}\left(E\right)}+1_{\sigma>1}\left|\nabla u\right|_{L_{p}\left(E\right)}\right)\right]
\end{align*}

\noindent where $K\left(\epsilon\right)\rightarrow0$ as $\epsilon\rightarrow0.$
Now we choose $\epsilon$ so that $NK\left(\epsilon\right)C\leq\frac{1}{2}$
and $K\left(\epsilon\right)C\le1,$ then
\[
\left|L^{\nu}u\right|_{L_{p}\left(E\right)}\leq2N\left[\left|f\right|_{L_{p}\left(E\right)}+\left(w\left(\epsilon\right)^{-1}+C\right)\left(\left|u\right|_{L_{p}\left(E\right)}+1_{\sigma>1}\left|\nabla u\right|_{L_{p}\left(E\right)}\right)\right]
\]

\noindent and 
\[
\left|u\right|_{L_{p}\left(E\right)}\leq\rho_{\lambda}\left(1+2N\right)\left[\left|f\right|_{L_{p}\left(E\right)}+\left(w\left(\epsilon\right)^{-1}+C\right)\left(\left|u\right|_{L_{p}\left(E\right)}+1_{\sigma>1}\left|\nabla u\right|_{L_{p}\left(E\right)}\right)\right].
\]

If $\sigma>1$, we apply interpolation inequality in Lemma \ref{lem:(Interpolation-Inequalities)}
(ii). For $\epsilon_{1}>0$ there is $C\left(\epsilon_{1}\right)>0$
independent of $u$ such that 
\[
\left|\nabla u\right|_{L_{p}\left(E\right)}\leq\epsilon_{1}\left|L^{\nu}u\right|_{L_{p}\left(E\right)}+C\left(\epsilon_{1}\right)\left|u\right|_{L_{p}\left(E\right)}.
\]

\noindent Thus by choosing $\epsilon_{1}$ sufficiently small so that
$2N\left(w\left(\epsilon\right)^{-1}+C\right)\epsilon_{1}\leq\frac{1}{2}$,
there is $N_{1}>0$ such that 

\noindent 
\[
\left|L^{\nu}u\right|_{L_{p}\left(E\right)}\leq N_{1}\left[\left|f\right|_{L_{p}\left(E\right)}+\left|u\right|_{L_{p}\left(E\right)}\right]
\]

\noindent and 
\[
\left|u\right|_{L_{p}\left(E\right)}\leq N_{1}\rho_{\lambda}\left[\left|f\right|_{L_{p}\left(E\right)}+\left|u\right|_{L_{p}\left(E\right)}\right].
\]

Finally, if we choose $\lambda>\lambda_{1}$ sufficiently large so
that $\frac{N_{1}}{\lambda}\leq\frac{1}{2}$ then there is $N_{2}>0$
such that 
\begin{align}
\left|L^{\nu}u\right|_{L_{p}\left(E\right)} & \leq N_{2}\left|f\right|_{L_{p}\left(E\right)}\label{eq:est_compact}\\
\left|u\right|_{L_{p}\left(E\right)} & \leq N_{2}\rho_{\lambda}\left|f\right|_{L_{p}\left(E\right)}.\nonumber 
\end{align}

\noindent Step II. (\cite[Lemma 9]{MPr14}) The estimate (\ref{eq:est_compact})
in Step I. is extended to any $u\in\tilde{C}_{p}^{\infty}\left(E\right)$
satisfying (\ref{eq:mainEq}). We let $\zeta\in C_{0}^{\infty}\left(\mathbf{R}^{d}\right)$
which has a compact support in a ball of radius $\epsilon$ centered
at $0$ then
\begin{equation}
\left|L^{\nu}u\left(t,x\right)\right|^{p}=\int\left|L^{\nu}u\left(t,x\right)\zeta\left(x-z\right)\right|^{p}dz\label{eq:zeta-0}
\end{equation}
\begin{align}
L^{\nu}u\left(t,x\right)\zeta\left(x-z\right) & =L^{\nu}\left[u\left(t,x\right)\zeta\left(x-z\right)\right]-u\left(t,x\right)L^{\nu}\zeta\left(x-z\right)\label{eq:zeta1}\\
 & +\int\left[u\left(t,x+y\right)-u\left(t,x\right)\right]\left[\zeta\left(x+y-z\right)-\zeta\left(x-z\right)\right]\nu\left(dy\right).\nonumber 
\end{align}

By using (\ref{eq:mainEq}), $u_{z}\left(t,x\right)=u\left(t,x\right)\zeta\left(x-z\right)$
must satisfies 
\begin{align*}
 & \partial_{t}\left(u_{z}\left(t,z\right)\right)\\
= & L^{m,\nu}u_{z}-\lambda u_{z}\left(t,z\right)+\zeta\left(x-z\right)f\left(t,x\right)-u\left(t,x\right)L^{m,\nu}\zeta\left(x-z\right)\\
- & \int\left[u\left(t,x+y\right)-u\left(t,x\right)\right]\left[\zeta\left(x+y-z\right)-\zeta\left(x-z\right)\right]m\left(t,x,y\right)\nu\left(dy\right)
\end{align*}

Since $u_{z}$ is compactly supported by a ball of radius $\epsilon$
centered at $z$, we may apply the estimate from Step I. By (\ref{eq:est_compact})
and Lemma \ref{lem:product0}, there exists $\delta\in\left(0,1\right)$
such that if $\lambda>\lambda_{1}$ ($\lambda_{1}$ from Step I) then
\begin{align*}
\int\left|L^{\nu}u_{z}\right|_{L_{p}\left(E\right)}^{p}dz & \leq N_{2}\left(\left|f\right|_{L_{p}\left(E\right)}^{p}+\left|u\right|_{L_{p}\left(E\right)}^{p}+\left|u\right|_{H_{p}^{\nu;\delta}\left(E\right)}^{p}\right)\\
\int\left|u_{z}\right|_{L_{p}\left(E\right)}^{p}dz & \leq N_{2}\rho_{\lambda}^{p}\left(\left|f\right|_{L_{p}\left(E\right)}^{p}+\left|u\right|_{L_{p}\left(E\right)}^{p}+\left|u\right|_{H_{p}^{\nu;\delta}\left(E\right)}^{p}\right).
\end{align*}

\noindent According to (\ref{eq:zeta-0}) and (\ref{eq:zeta1}), there
is $N_{2}>0$ such that 
\begin{align*}
\left|L^{\nu}u\right|_{L_{p}\left(E\right)}^{p} & \leq N_{2}\left(\left|f\right|_{L_{p}\left(E\right)}^{p}+\left|u\right|_{L_{p}\left(E\right)}^{p}+\left|u\right|_{H_{p}^{\nu;\delta}\left(E\right)}^{p}\right)\\
\left|u\right|_{L_{p}\left(E\right)}^{p} & \leq N_{2}\rho_{\lambda}^{p}\left(\left|f\right|_{L_{p}\left(E\right)}^{p}+\left|u\right|_{L_{p}\left(E\right)}^{p}+\left|u\right|_{H_{p}^{\nu;\delta}\left(E\right)}^{p}\right).
\end{align*}

Applying the interpolation inequality in Lemma \ref{lem:(Interpolation-Inequalities)}
(i) with $\epsilon>0$ and again choosing sufficiently small $\epsilon$
and sufficiently large $\lambda>\lambda_{2}>\lambda_{1}$, then there
is $N_{3}>0$ such that
\begin{align}
\left|L^{\nu}u\right|_{L_{p}\left(E\right)} & \leq N_{3}\left|f\right|_{L_{p}\left(E\right)}\label{eq:est_non_compact}\\
\left|u\right|_{L_{p}\left(E\right)} & \leq N_{3}\rho_{\lambda}\left|f\right|_{L_{p}\left(E\right)}.\nonumber 
\end{align}

\noindent Step III. (\cite[Corollary 3]{MPr14}) The estimate (\ref{eq:est_non_compact})
is extended to all $\lambda\geq0$ exactly with the same argument
as in the original proof. Note that the bounding constant becomes
dependent on $T.$ 

\noindent Step IV. (\cite[Proof of Theorem 1]{MPr14}) The proof is
completed by using apriori estimate from Step III. and a continuation
of parameter argument for the operator 
\[
M_{\tau}u=\tau L^{m,\nu}u+\left(1-\tau\right)L^{\nu}u,\hspace{1em},u\in\mathcal{H}_{p}^{\nu}\left(E\right),\tau\in\left[0,1\right].
\]

Note that by \cite[Theorem 1]{MPh3} the mapping $T_{\tau}$ given
by $u\left(t,x\right)=\int_{0}^{t}F\left(s,x\right)ds\rightarrow F-M_{\tau}u$
is an onto map from $\mathcal{H}_{p}^{\nu}\left(E\right)\rightarrow L_{p}\left(E\right)$
when $\tau=0.$ Owing to \cite[Theorem 5.2]{GT}, $T_{\tau}$ is an
onto map for any $\tau\in\left[0,1\right]$ if we can show that $\left|T_{\tau}u\right|_{L_{p}\left(E\right)}\asymp\left|u\right|_{\mathcal{H}_{p}^{\nu}\left(E\right)}$
with constants independent of $\tau.$

First, we show that 
\begin{equation}
\left|L^{m,\nu}u\right|_{L_{p}\left(E\right)}\leq N\left|u\right|_{H_{p}^{\nu}\left(E\right)}\label{eq:Lmnu <=00003D u}
\end{equation}
 by considering $L^{m,\nu}$ in (\ref{eq:zeta-0}) and (\ref{eq:zeta1})
instead of $L^{\nu}$, in fact,
\[
\left|L^{m,\nu}u\left(t,x\right)\right|^{p}=\int\left|L^{m,\nu}u\left(t,x\right)\zeta\left(x-z\right)\right|^{p}dz
\]
\begin{align*}
L^{m,\nu}u\left(t,x\right)\zeta\left(x-z\right) & =L^{m,\nu}\left[u\left(t,x\right)\zeta\left(x-z\right)\right]-u\left(t,x\right)L^{m,\nu}\zeta\left(x-z\right)\\
 & +\int\left[u\left(t,x+y\right)-u\left(t,x\right)\right]\left[\zeta\left(x+y-z\right)-\zeta\left(x-z\right)\right]m\left(t,y\right)\nu\left(dy\right)\\
 & =A_{1}+A_{2}+A_{3}.
\end{align*}

\noindent Following \cite[Lemma 5.2]{DK}, $A_{2}$ is estimated directly
by Minkowski inequality, $A_{3}$ is estimated by Lemma \ref{lem:product0}
and $A_{1}$ is estimated by Lemma \ref{lem:emb2} (ii).

Consequently, by (\ref{eq:Lmnu <=00003D u}), 

\begin{align*}
\left|T_{\tau}u\right|_{L_{p}\left(E\right)} & =\left|F-M_{\tau}u\right|_{L_{p}\left(E\right)}\leq\left|F\right|_{L_{p}\left(E\right)}+N\left|u\right|_{H_{p}^{\nu}\left(E\right)}\\
 & \leq N\left|u\right|_{\mathcal{H}_{p}^{\nu}\left(E\right)}.
\end{align*}

\noindent where $N$ is independent of $\tau.$

For the reverse inequality, we write
\[
u\left(t,x\right)=\int_{0}^{t}F\left(s,x\right)ds=\int_{0}^{t}T_{\tau}u\left(s,x\right)+M_{\tau}u\left(s,x\right)ds.
\]

\noindent By (\ref{eq:est_non_compact}), 
\begin{equation}
\left|u\right|_{H_{p}^{\nu}\left(E\right)}\leq N\left|T_{\tau}u\right|_{L_{p}\left(E\right)}=N\left|F-M_{\tau}u\right|_{L_{p}\left(E\right)}.\label{eq:last}
\end{equation}
By (\ref{eq:Lmnu <=00003D u}) and (\ref{eq:last}), 
\[
\left|M_{\tau}u\right|_{L_{p}\left(E\right)}\leq N\left|u\right|_{H_{p}^{\nu}\left(E\right)}\leq N\left|F-M_{\tau}u\right|_{L_{p}\left(E\right)}.
\]

\noindent Hence, 
\begin{align*}
\left|u\right|_{\mathcal{H}_{p}^{\nu}\left(E\right)} & =\left|u\right|_{H_{p}^{\nu}\left(E\right)}+\left|F\right|_{L_{p}\left(E\right)}\leq N\left(\left|F-M_{\tau}u\right|_{L_{p}\left(E\right)}+\left|M_{\tau}u\right|_{L_{p}\left(E\right)}\right)\\
 & \leq N\left|F-M_{\tau}u\right|_{L_{p}\left(E\right)}=N\left|T_{\tau}u\right|_{L_{p}\left(E\right)}.
\end{align*}

\noindent where $N$ is independent of $\tau.$ 

\noindent The proof is finished. 
\end{proof}

\section{Appendix}

We collect some results on O-RV functions which are used frequently
in this paper. Some results are taken from \cite{MF,MPh3} and slightly
modified for our need. 

\noindent 
\begin{lem}
\label{lem:alpha12_ratio}(\cite[Lemma 1]{MPh3}) Assume $w\left(r\right),r>0,$
is an O-RV function at zero and infinity with indices $p_{1},q_{1},p_{2},q_{2}$
defined in (\ref{1}), (\ref{2}), and $p_{1},p_{2}>0.$ Then for
any $\alpha_{1}>q_{1}\vee q_{2}$ , $0<\alpha_{2}<p_{1}\wedge p_{2}$,
there exist $c_{1}=c_{1}\left(\alpha_{1}\right),c_{2}=c_{2}\left(\alpha_{2}\right)>0$
such that 
\[
c_{1}\left(\frac{y}{x}\right)^{\alpha_{2}}\leq\frac{w\left(y\right)}{w\left(x\right)}\leq c_{2}\left(\frac{y}{x}\right)^{\alpha_{1}},0<x\leq y<\infty.
\]
\end{lem}
\begin{lem}
\noindent \label{lem:al1}(\cite[Lemma 8]{MF}) Assume $w\left(r\right),r>0,$
is an O-RV function at zero with lower and upper indices $p_{1},q_{1},$
that is, 
\[
r_{1}\left(x\right)=\overline{\lim_{\varepsilon\rightarrow0}}\frac{w\left(\epsilon x\right)}{w\left(\epsilon\right)}<\infty,x>0,
\]
and 
\[
p_{1}=\lim_{\epsilon\rightarrow0}\frac{\log r_{1}\left(\epsilon\right)}{\log\epsilon}\leq q_{1}=\lim_{\epsilon\rightarrow\infty}\frac{\log r_{1}\left(\epsilon\right)}{\log\left(\epsilon\right)}.
\]
a) Let $\beta>0$ and $\tau>-\beta p_{1}.$ There is $C>0$ so that
\[
\int_{0}^{x}t^{\tau}w\left(t\right)^{\beta}\frac{dt}{t}\leq Cx^{\tau}w\left(x\right)^{\beta},x\in(0,1],
\]
and $\lim_{x\rightarrow0}x^{\tau}w\left(x\right)^{\beta}=0.$

\noindent b) Let $\beta>0$ and $\tau<-\beta q_{1}$. There is $C>0$
so that 
\[
\int_{x}^{1}t^{\tau}w\left(t\right)^{\beta}\frac{dt}{t}\leq Cx^{\tau}w\left(x\right)^{\beta},x\in(0,1],\text{ }
\]
and $\lim_{x\rightarrow0}x^{\tau}w\left(x\right)^{\beta}=\infty.$

\noindent c) Let $\beta<0$ and $\tau>-\beta q_{1}$. There is $C>0$
so that 
\[
\int_{0}^{x}t^{\tau}w\left(t\right)^{\beta}\frac{dt}{t}\leq Cx^{\tau}w\left(x\right)^{\beta},x\in(0,1],
\]
and $\lim_{x\rightarrow0}x^{\tau}w\left(x\right)^{\beta}=0.$

\noindent d) Let $\beta<0$ and $\tau<-\beta p_{1}$. There is $C>0$
so that 
\[
\int_{x}^{1}t^{\tau}w\left(t\right)^{\beta}\frac{dt}{t}=\int_{1}^{x^{-1}}t^{-\tau}w\left(\frac{1}{t}\right)^{\beta}\frac{dt}{t}\leq Cx^{\tau}w\left(x\right)^{\beta},x\in(0,1],
\]
and $\lim_{x\rightarrow0}x^{\tau}w\left(x\right)^{\beta}=\infty.$ 
\end{lem}
\noindent Similar statement holds for O-RV functions at infinity.
\begin{lem}
\noindent \label{lem:al2}(\cite[Lemma 3]{MPh3}) Assume $w\left(r\right),r>0,$
is an O-RV function at infinity with lower and upper indices $p_{2},q_{2},$
that is, 
\[
r_{2}\left(x\right)=\overline{\lim_{\varepsilon\rightarrow\infty}}\frac{w\left(\epsilon x\right)}{w\left(\epsilon\right)}<\infty,x>0,
\]
and 
\[
p_{2}=\lim_{\epsilon\rightarrow0}\frac{\log r_{2}\left(\epsilon\right)}{\log\epsilon}\leq q_{2}=\lim_{\epsilon\rightarrow\infty}\frac{\log r_{2}\left(\epsilon\right)}{\log\left(\epsilon\right)}.
\]
a) Let $\beta>0$ and $-\tau>\beta q_{2}.$ There is $C>0$ so that
\[
\int_{x}^{\infty}t^{\tau}w\left(t\right)^{\beta}\frac{dt}{t}\leq Cx^{\tau}w\left(x\right)^{\beta},x\in\lbrack1,\infty),
\]
and $\lim_{x\rightarrow\infty}x^{\tau}w\left(x\right)^{\beta}=0.$

\noindent b) Let $\beta>0$ and $-\tau<\beta p_{2}$. There is $C>0$
so that 
\[
\int_{1}^{x}t^{\tau}w\left(t\right)^{\beta}\frac{dt}{t}\leq Cx^{\tau}w\left(x\right)^{\beta},x\in\lbrack1,\infty),\text{ }
\]
and $\lim_{x\rightarrow\infty}x^{\tau}w\left(x\right)^{\beta}=\infty.$

\noindent c) Let $\beta<0$ and $\tau<-\beta p_{2}$. There is $C>0$
so that 
\[
\int_{x}^{\infty}t^{\tau}w\left(t\right)^{\beta}\frac{dt}{t}\leq Cx^{\tau}w\left(x\right)^{\beta},x\in\lbrack1,\infty),
\]
and $\lim_{x\rightarrow\infty}x^{\tau}w\left(x\right)^{\beta}=0.$

\noindent d) Let $\beta<0$ and $\tau>-\beta q_{2}$. There is $C>0$
so that 
\[
\int_{1}^{x}t^{\tau}w\left(t\right)^{\beta}\frac{dt}{t}\leq Cx^{\tau}w\left(x\right)^{\beta},x\in\lbrack1,\infty),
\]
and $\lim_{x\rightarrow\infty}x^{\tau}w\left(x\right)^{\beta}=\infty.$ 
\end{lem}
\begin{lem}
\noindent \label{lem:index_a}(\cite[Lemma 5 ]{MPh3}) Let $w\left(r\right),r>0,$
be a continuous, non-decreasing O-RV function at zero and infinity
with indices $p_{1},q_{1},p_{2},q_{2}$ defined in (\ref{1}), (\ref{2}),
and $p_{1},p_{2}>0.$ Let 
\[
a_{1}\left(r\right)=\inf\left\{ t>0:w\left(t\right)\geq r\right\} ,r>0.
\]
\[
a_{2}\left(r\right)=\inf\left\{ t>0:w\left(t\right)>r\right\} ,r>0.
\]

\noindent If $a$ is either $a_{1}$ or $a_{2}$, then 

\noindent (i) $w\left(a\left(t\right)\right)=t,t>0$ and 
\[
a_{1}\left(w\left(t\right)\right)\leq t\leq a_{1}\left(w\left(t\right)+\right),t>0,
\]
\begin{align*}
a_{2}\left(w\left(t\right)-\right) & \leq t\leq a_{2}\left(w\left(t\right)\right),t>0.
\end{align*}

\noindent (ii) $a$ is O-RV at zero and infinity with lower indices
$p,\bar{p}$ and upper indices $q,\bar{q}$ respectively such that
\[
\frac{1}{q_{1}}\leq p\leq q\leq\frac{1}{p_{1}},\frac{1}{q_{2}}\leq\bar{p}\leq\bar{q}\leq\frac{1}{p_{2}}.
\]

\noindent (iii) $a\left(w\left(t\right)\right)\asymp t,t>0.$
\end{lem}
\begin{proof}
(i) is straightforward, (ii) is proved in \cite[Lemma 5 ]{MPh3} and
(iii) easily follows from (ii) and Lemma \ref{lem:alpha12_ratio}.
\end{proof}
\begin{cor}
\noindent \label{cor:aymp_a}(\cite[Corollary 1]{MPh3}) Let $w\left(r\right),r>0,$
be a continuous, non-decreasing O-RV function at zero and infinity
with indices $p_{1},q_{1},p_{2},q_{2}$ defined in (\ref{1}), (\ref{2}),
and $p_{1},p_{2}>0.$ Let 
\[
a_{1}\left(r\right)=\inf\left\{ t>0:w\left(t\right)\geq r\right\} ,r>0.
\]
\[
a_{2}\left(r\right)=\inf\left\{ t>0:w\left(t\right)>r\right\} ,r>0.
\]

\noindent If $a$ is either $a_{1}$ or $a_{2}$, then 

\noindent (i) For any $\beta>0$ and $\tau<\frac{\beta}{q_{1}}\wedge\frac{\beta}{q_{2}}$
there is $C>0$ such that 
\begin{eqnarray*}
\int_{0}^{r}t^{-\tau}a\left(t\right)^{\beta}\frac{dt}{t} & \leq & Cr^{-\tau}a\left(r\right)^{\beta},r>0,\\
\lim_{r\rightarrow0}r^{-\tau}a\left(r\right)^{\beta} & = & 0,\lim_{r\rightarrow\infty}r^{-\tau}a\left(r\right)^{\beta}=\infty,
\end{eqnarray*}
and for any $\beta<0,\tau>\left(-\frac{\beta}{p_{1}}\right)\vee\left(-\frac{\beta}{p_{2}}\right)$
there is $C>0$ such that 
\begin{eqnarray*}
\int_{0}^{r}t^{\tau}a\left(t\right)^{\beta}\frac{dt}{t} & \leq & Cr^{\tau}a\left(r\right)^{\beta},r>0,\\
\lim_{r\rightarrow0}r^{\tau}a\left(r\right)^{\beta} & = & 0,\lim_{r\rightarrow\infty}r^{\tau}a\left(r\right)^{\beta}=\infty.
\end{eqnarray*}
(ii) For any $\gamma>0$ and $\tau>\frac{\gamma}{p_{1}}\vee\frac{\gamma}{p_{2}}$
there is $C>0$ such that 
\begin{eqnarray*}
\int_{r}^{\infty}t^{-\tau}a\left(t\right)^{\gamma}\frac{dt}{t} & \leq & Cr^{-\tau}a\left(r\right)^{\gamma},r>0,\\
\lim_{r\rightarrow0}r^{-\tau}a\left(r\right)^{\gamma} & = & \infty,\lim_{r\rightarrow\infty}r^{-\tau}a\left(r\right)^{\gamma}=0,
\end{eqnarray*}
and for any $\gamma<0$ and $\tau<\left(-\frac{\gamma}{q_{1}}\right)\wedge\left(-\frac{\gamma}{q_{2}}\right)$
there is $C>0$ such that 
\begin{eqnarray*}
\int_{r}^{\infty}t^{\tau}a\left(t\right)^{\gamma}\frac{dt}{t} & \leq & Cr^{\tau}a\left(r\right)^{\gamma},r>0,\\
\lim_{r\rightarrow0}r^{\tau}a\left(r\right)^{\gamma} & = & \infty,\lim_{r\rightarrow\infty}r^{\tau}a\left(r\right)^{\gamma}=0.
\end{eqnarray*}
\end{cor}
\noindent 
\begin{lem}
\noindent \label{lem:gen-by-part} Let $\nu\in\mathfrak{A}^{\sigma}$,$\sigma\in\left[1,2\right)$,
$w=w_{\nu}$ be an O-RV function at zero and infinity with indices
$p_{1},q_{1},p_{2},q_{2}$ defined in (\ref{1}), (\ref{2}) and $p_{1},p_{2}>0$.
If $\frac{\sigma-1}{\left(p_{1}\wedge p_{2}\right)}<\beta<1$ then
there exists $N=N\left(\beta,\nu\right)>0$ such that
\[
\int_{\left|y\right|\leq1}\left|y\right|w\left(\left|y\right|\right)^{\beta}\nu\left(dy\right)\leq N.
\]
\end{lem}
\begin{proof}
Since $1+\left(p_{1}\wedge p_{2}\right)\beta>\sigma$, we may choose
$0<\alpha<\left(p_{1}\wedge p_{2}\right)$ so that $1+\alpha\beta>\sigma.$
By Lemma \ref{lem:alpha12_ratio} and definition of $\sigma$, 
\begin{align*}
\int_{\left|y\right|\leq1}\left|y\right|w\left(\left|y\right|\right)^{\beta}\nu\left(dy\right) & \leq N\int_{\left|y\right|\leq1}\left|y\right|^{1+\alpha\beta}\nu\left(dy\right)\leq N.
\end{align*}
\end{proof}
\noindent 
\begin{lem}
\noindent \label{lem:w=00003Dgamma} Let $\nu\left(dy\right)=j_{d}\left(\left|y\right|\right)dy$
and \textbf{D }hold then 

\noindent 
\[
w_{\nu}\left(r\right)\asymp\left(\int_{r}^{\infty}\frac{1}{\gamma\left(s\right)s}ds\right)^{-1}\asymp\gamma\left(r\right),r>0.
\]
\end{lem}
\begin{proof}
\noindent The first estimate follows from the definition of $w_{\nu}.$
The second estimate follows from basic properties of O-RV functions.
With definition of O-RV functions in mind, we consider the function
$\gamma^{-1}$, 
\[
\limsup_{\epsilon\rightarrow\infty}\frac{\gamma\left(\epsilon x\right)^{-1}}{\gamma\left(\epsilon\right)^{-1}}=\limsup_{\epsilon\rightarrow\infty}\frac{\gamma\left(\epsilon\right)}{\gamma\left(\epsilon x\right)}=\limsup_{\epsilon\rightarrow\infty}\frac{\gamma\left(x^{-1}\epsilon x\right)}{\gamma\left(\epsilon x\right)}=r_{2}\left(x^{-1}\right).
\]

\noindent Hence, $\gamma{}^{-1}$ is an O-RV function at infinity
with the lower index of 

\noindent 
\[
p=\lim_{\epsilon\rightarrow0}\frac{\log r_{2}\left(\epsilon^{-1}\right)}{\log\epsilon}=-\lim_{\epsilon\rightarrow\infty}\frac{\log r_{2}\left(\epsilon\right)}{\log\epsilon}=-q_{2}
\]
and by a similar calculation the upper index is 
\[
q=-p_{2}.
\]

Similarly, we consider the function $\gamma\left(x^{-1}\right)^{-1},x>0,$
\[
\limsup_{\epsilon\rightarrow\infty}\frac{\gamma^{-1}\left(\frac{1}{\epsilon x}\right)}{\gamma^{-1}\left(\frac{1}{\epsilon}\right)}=\limsup_{\epsilon\rightarrow0}\frac{\gamma\left(\epsilon\right)}{\gamma\left(\epsilon x^{-1}\right)}=\limsup_{\epsilon\rightarrow0}\frac{\gamma\left(\epsilon x^{-1}x\right)}{\gamma\left(\epsilon x^{-1}\right)}=r_{2}\left(x\right).
\]

\noindent Therefore, the function $\gamma\left(x^{-1}\right)^{-1},x>0$
is an O-RV function at infinity with 
\[
p=p_{2},q=q_{2}.
\]

\noindent By applying \cite[Theorem 3]{AA} with $\gamma{}^{-1}$,
there exists $N>1$ so that for $r\geq N$,
\begin{equation}
\int_{r}^{\infty}\frac{1}{\gamma\left(s\right)s}ds\asymp\gamma\left(r\right)^{-1}.\label{eq:r>=00003DN}
\end{equation}

Now we extend the estimate to $r\in\left[1,N\right)$. On one hand,
by (\ref{eq:r>=00003DN}) and Lemma \ref{lem:alpha12_ratio},
\begin{align*}
\int_{r}^{\infty}\frac{1}{\gamma\left(s\right)s}ds & =\int_{r}^{N}\frac{1}{\gamma\left(s\right)s}ds+\int_{N}^{\infty}\frac{1}{\gamma\left(s\right)s}ds\\
 & \leq C\gamma\left(r\right)^{-1}\ln\left(N\right)+C\gamma\left(N\right)^{-1}\\
 & \leq C\gamma\left(r\right)^{-1}.
\end{align*}

\noindent On the other hand, by Lemma \ref{lem:alpha12_ratio},
\begin{align*}
\int_{r}^{\infty}\frac{1}{\gamma\left(s\right)s}ds & \geq\int_{N}^{\infty}\frac{1}{\gamma\left(s\right)s}ds\geq C\gamma\left(N\right)^{-1}\\
 & \geq C\gamma\left(r\right)^{-1}.
\end{align*}

\noindent Therefore,
\begin{equation}
\int_{r}^{\infty}\frac{1}{\gamma\left(s\right)s}ds\asymp\gamma\left(r\right)^{-1},r\geq1.\label{eq:r>=00003D1}
\end{equation}

For $r<1$, we first note that by applying \cite[Theorem 3]{AA} with
$\gamma\left(x^{-1}\right)^{-1},x>0$, there exists $N>1$ so that
for $x\geq N$. 
\begin{equation}
\int_{1}^{x}\frac{1}{\gamma\left(s^{-1}\right)s}ds\asymp\gamma\left(x^{-1}\right)^{-1}.\label{eq:l_2}
\end{equation}

For $0<r\leq\frac{1}{N}$, we change the variable of integration then
apply (\ref{eq:l_2}), (\ref{eq:r>=00003D1}) and Lemma \ref{lem:alpha12_ratio},

\noindent 
\begin{align*}
\int_{r}^{\infty}\frac{1}{\gamma\left(s\right)s}ds & =\int_{r}^{1}\frac{1}{\gamma\left(s\right)s}ds+\int_{1}^{\infty}\frac{1}{\gamma\left(s\right)s}ds\\
 & =\int_{1}^{1/r}\frac{1}{\gamma\left(s^{-1}\right)s}ds+\int_{1}^{\infty}\frac{1}{\gamma\left(s\right)s}ds\\
 & \leq C\gamma\left(r\right)^{-1}.
\end{align*}

\noindent Moreover, 
\[
\int_{r}^{\infty}\frac{1}{\gamma\left(s\right)s}ds\geq\int_{1}^{1/r}\frac{1}{\gamma\left(s^{-1}\right)s}ds\geq C\gamma\left(r\right)^{-1}.
\]

\noindent Therefore,
\begin{equation}
\int_{r}^{\infty}\frac{1}{\gamma\left(s\right)s}ds\asymp\gamma\left(r\right)^{-1},r\in\left(0,\frac{1}{N}\right].\label{eq:r<=00003D1/N}
\end{equation}

\noindent Finally, the estimate for $r\in\left[\frac{1}{N},1\right)$
follows from (\ref{eq:r>=00003D1}), (\ref{eq:r<=00003D1/N}) and
Lemma \ref{lem:alpha12_ratio}.
\end{proof}

\end{document}